\documentclass[a4paper,12pt]{article}

\usepackage[utf8]{inputenc}
\usepackage[english]{babel}
\usepackage{enumitem}
\usepackage{float}
\usepackage{amsmath}
\usepackage{amssymb}
\usepackage{amsthm}
\usepackage{amsfonts}
\usepackage{booktabs}
\usepackage{epsfig}
\usepackage{epstopdf}
\usepackage{subfigure}
\usepackage{cite}
\usepackage{multirow} % multirow for format of table
\usepackage{algorithm,algpseudocode}
\usepackage{subcaption}
\usepackage{xcolor}

\usepackage[linkcolor=blue,colorlinks=true]{hyperref}

\newtheorem{definition}{Definition}
\newtheorem{theorem}{Theorem}
\newtheorem{lemma}[theorem]{Lemma}

\usepackage{caption}
\newtheorem{dfn}{Definition}[section]

\newtheorem{exm}[dfn]{Example}
\theoremstyle{definition}

\newcommand{\eaa}[1]{\color{black}#1\normalcolor}
\allowdisplaybreaks

\title{A Robust Numerical Scheme for Solving Riesz-Tempered Fractional Reaction-Diffusion Equations}

\date{\today}

\author{Mohammad Partohaghighi\footnote{Department of Applied Mathematics, University of California Merced, Merced, CA, USA; e-mail: mpartohaghighi@ucmerced.edu}
\hspace*{0.8mm}, Emmanuel Asante-Asamani\footnote{Department of Mathematics, Clarkson University, Potsdam, NY, USA; easantea@clarkson.edu}
\hspace*{0.8mm}, Olaniyi S. Iyiola\footnote{Department of Mathematics, Morgan State University, Baltimore, MD, USA; e-mail:niyi4oau@gmail.com; olaniyi.iyiola@morgan.edu}}

\begin{document}

\maketitle

\begin{abstract}
\noindent The Fractional Diffusion Equation (FDE) is a mathematical model that describes anomalous transport phenomena characterized by non-local and long-range dependencies which deviate from the traditional behavior of diffusion. Solving this equation numerically is challenging due to the need to discretize complicated integral operators which increase the computational costs. These complexities are exacerbated by nonlinear source terms, nonsmooth data and irregular domains. In this study, we propose a second order Exponential Time Differencing Finite Element Method (ETD-RDP-FEM) to efficiently solve nonlinear FDE, posed in irregular domains. This approach discretizes matrix exponentials using a rational function with real and distinct poles, resulting in an L-stable scheme that damps spurious oscillations caused by non-smooth initial data. The method is shown to outperform existing second-order methods for FDEs with a higher accuracy and faster computational time.\\


\noindent {\bf 2010 MSC classification:} 

\end{abstract}

\section{Introduction}\label{Sec:Intro}
\noindent Fractional calculus is a developing field in mathematics that explores non-integer order derivatives and integrals. Recently, this area has attracted significant attention due to its broad range of applications in various scientific areas including chemistry, biology, engineering, and finance, among others \cite{fractionalapplications1,fractionalapplications2,fractionalapplications3,fractionalapplications4,fractionalapplications5,fractionalapplications6,fractionalapplications7,iyi1}. One of the main drivers for the study of fractional calculus is its ability to provide models of physical phenomena that are difficult to explain using classical calculus. For example, it has been used to model viscoelastic materials \cite{viscoelasticmaterials} which show time-dependent deformation and stress relaxation as well as to describe particle diffusion in complex media \cite{complexmedia}. Caputo, Riemann-Liouville, Reisz, Grünwald-Letnikov, Tempered fractional derivatives are among the several types of non-integer order derivatives with specific associated applications and/or properties, see \cite{iyi3,iyi4}.\\

\noindent The fractional diffusion equation \eaa{describes the} behavior of particles undergoing anomalous diffusion and/or \eaa{chemical reactions}. This phenomenon is characterized by a non-linear correlation between mean squared displacement and time. Anomalous diffusion is prevalent in various physical, biological, and chemical systems including porous media, polymers, and biological tissues \cite{tissues1,tissues2,tissues3,iyi2}. The tempered fractional diffusion equation (TFDE) adjusts the standard second \eaa{order} spatial derivative into a tempered fractional derivative. This adjustment refines the parameters of random walk models ruled by an exponentially tempered power-law jump distribution. In finance, one typically observes the creation of semi-heavy tails due to the restricted tempered stable probability densities. Tempered fractional time derivatives emerge from waiting times governed by a tempered power law, which have proven significantly valuable in geophysics. When it comes to Brownian motion, the associated tempered fractional derivative or integral, otherwise known as tempered fractional Brownian motion, can exhibit semi-long range dependence \cite{TFDE1}.\\

Numerical \eaa{solution of} TFDE comes with several notable challenges. The first is the so-called curse of non-locality. The characteristic of non-locality in fractional diffusion equations stems from the employment of fractional derivatives. These are an extension of conventional derivatives but to non-integer orders. Consequently, unlike the typical derivatives that depend only on the function's behavior in the immediate vicinity of a specific point, fractional derivatives rely on the function's characteristics across the entire domain. This implies that the value of a fractional derivative at a given point is influenced by the function's behavior everywhere, not just near that particular point. Another layer of complexity comes from the memory effect of the tempered fractional operator, where the future state of the system is dependent on all of its past states. The memory effect implies long-range dependence, whereby states from the distant past can still significantly influence the current state. This feature complicates the truncation method commonly used in numerical solutions to limit the influence of previous states, making the system far more intricate to solve. The third challenge lies in nonlinearity \eaa{in the reaction term}, which \eaa{typically requires costly iterative methods resulting from an implicit discretization of temporal derivatives}. The fourth challenge, non smoothness \eaa{in the initial data}, \eaa{results in spurious oscillations which can blow up the numerical solution if not sufficiently damped. Finally, TFDEs posed in irregular spatial domains are usually difficult to discretize with finite difference schemes.}\\

Several strategies have been implemented to \eaa{solve the} tempered fractional diffusion equation. \eaa{In  \cite{TFDE3} a Forward Euler (FE) method is combined with a weighted and shifted Lubich difference (WSLD) scheme to solve a linear FTDE. The use of an explicit FE method introduces a stability restriction on the time step with can lead to costly simulations. } \eaa{In \cite{TFDE2}} the Crank-Nicolson (CN) method is combined with a Finite Element Method (FEM) \eaa{to resolve the challenge with solving TFDE in irregular domains and remove the stability restriction on the time step}. However, the study did not investigate the nonlinear problem, \eaa{which is likely to be costly since it will require an iterative scheme to evolve the implicit CN method}. \eaa{The finite element method as been applied to solve other fractional boundary value problems involving Riemann-Liouville or Caputo derivative \cite{fem1},  fractional advection-diffusion reaction equations\cite{fem2}, and fractional diffusion-wave equation \cite{fem3}}. \\

% as well as  For instance, Jin et al. devised variational formulations of the Petrov--Galerkin type for one-dimensional fractional boundary value problems that incorporate either a Riemann-Liouville or Caputo derivative in the principal term, along with convection and potential terms. Similarly, Lian et al. provided an in-depth numerical analysis of spatial fractional advection-diffusion equations (FADE) using the finite element method (FEM). They showed that a standard Galerkin finite element formulation of the pure fractional diffusion equation, without advection, can result in numerical oscillations in the solution based on the fractional derivative order. In another study, Chakraborty et al.  presented an ADI Galerkin finite element scheme for the two-dimensional time fractional diffusion-wave equation.

\eaa{The family of }Exponential Time Differencing (ETD) schemes \eaa{avoid costly iterative procedures such as Newton's method by utilizing an implicit discretization of the stiff, typically linear, terms in a differential equation (DE) and an explicit treatment of the nonstiff, typically nonlinear, terms to achieve excellent stability properties at a reduced computational cost\cite{Cox2002,kassam2005fourth}. In \cite{asanRDPP1,asanRDPP2} a second order L-stable exponential time differencing (ETD-RDP) method is introduced for non-linear equations of reaction-diffusion type (RDE) and is shown to efficiently damp out spurious oscillations resulting from nonsmooth initial data. The method was initially applied to integer RDE and later extended to fractional RDE in \cite{asanRDPP3,asanRDPP5}.}\\

In this paper, we consider the following tempered fractional diffusion equation with the order $\alpha$
\begin{equation}
\begin{split}
 u_t(x,t)-\partial_{\mid x \mid}^{\alpha ,\lambda}u(x,t)&=f(x,t,u(x,t)), \label{MainEquation}\\ 
\end{split}
\end{equation}
 where $$ \partial_{\mid x \mid}^{\alpha , \lambda}u(x,t)=C_{\alpha}\bigg(\partial_{-x}^{\alpha ,\lambda}u(x,t)+\partial_{+x}^{\alpha,\lambda}u(x,t)\bigg), \ C_{\alpha} =  \frac{1}{2cos(0.5\pi)},$$
with the initial and boundary conditions
\begin{equation}\label{BC}
\begin{split}
u(0,t)&=u(1,t)=0,\\
u(x,0)&=g(x),
\end{split}
\end{equation}
where $\lambda$ is the tempering parameter, the source function here is $f(x,t,u(x,t))$, and the derivatives $\partial_{-x}^{\alpha ,\lambda}u(x,t)$ and $\partial_{+x}^{\alpha,\lambda}u(x,t)$ are defined as follows
\begin{align*}
\partial_{+x}^{\alpha,\lambda}u(x,t)&=D_{+x}^{\alpha , \lambda}u(x,t)-\lambda^{\alpha}u(x,t)-\alpha \lambda^{\alpha -1}D_xu(x,t)\\
\\
\partial_{-x}^{\alpha,\lambda}u(x,t)&=D_{-x}^{\alpha , \lambda}u(x,t)-\lambda^{\alpha}u(x,t)+\alpha \lambda^{\alpha -1}D_xu(x,t),
\end{align*}
\\
also, $D_{+x}^{\alpha , \lambda}u(x,t)$ and $D_{-x}^{\alpha , \lambda}u(x,t)$ denote the left and right tempered fractional derivatives of the function $u(x,t)$, respectively and $D_x$ is the general derivatie operator. We propose to utilize a novel combination of Exponential Time Differencing (ETD) together with Finite Element Method (FEM) to solve the tempered fractional diffusion equation. \eaa{We adopt a method of lines approach and begin with an FEM discretization of the model in space to yield a system of ordinary differential equations which are then solved in time with the ETDRDP scheme.}

\textbf{Our contribution.} This paper introduces an innovative strategy, the Exponential Time Differencing Finite Element Method (ETDRDP-FEM), which unites the Finite Element Method (FEM) and Exponential Time Differencing (ETDRDP) \eaa{scheme to solve the nonlinear TFDE}. This combined approach is designed to solve various \eaa{tempered } fractional models efficiently, encompassing linear, nonlinear, and nonsmooth problems \eaa{in regular or irregular domains}. By amalgamating the strengths of both FEM and ETDRDP, the method we propose aspires to offer superior accuracy and computational efficacy in tackling intricate fractional problems, whilst preserving its versatility across a broad spectrum of situations.\\

\textbf{Outline.}
The organization of this paper is as follows: Section \ref{Sec:Prelims} introduces the essential concepts and primary outcomes of fractional calculus \eaa{necessary for deriving the FEM method}. The numerical discretizations of the fractional diffusion equation, which includes the variational formulation of the primary equation and its corresponding discretizations in \eaa{space and time}, are detailed in Section \ref{Sec:main}. Section \ref{sec:numerics} delves into the numerical experiments performed. Finally, Section \ref{conclude} wraps up the paper with concluding observations.
\section{Preliminaries}\label{Sec:Prelims}
\noindent
This section is devoted to recall some definitions and basic results necessary to derive the variational formulation of the TFDE.
\begin{definition}
For any $\alpha>0$, gamma function, denoted as $\Gamma$, is defined as
$$\Gamma(\alpha) = \int_{0}^{\infty} \eta^{\alpha - 1} e^{-\eta} d\eta,$$
with the property of $\Gamma(\alpha + 1) = \alpha \Gamma(\alpha).$
\end{definition}
\begin{definition}
For $\operatorname{Re}(\alpha) > 0$ and $\operatorname{Re}(\beta) > 0$, Beta function, denoted as $B$, is defined as \cite{Gammafunction}
$$B(\alpha, \beta) = \int_{0}^{1} \eta^{\alpha - 1} (1 - \eta)^{\beta - 1} d\eta.$$ 
\end{definition}
The beta function is closely related to the gamma function, and it can be expressed in terms of gamma function using the formula: 
\begin{equation}
    B(\alpha, \beta) = \frac{\Gamma(\alpha) \Gamma(\beta)}{\Gamma(\alpha + \beta)}. \label{betaproperty}
\end{equation}
\begin{definition}[Riemann-Liouville Integral] \cite{riemannliouvilleintegrallll}
Let $f \in L^1[a,b]$ be a real-valued locally integrable function. Then for $-\infty \leq a < w < b \leq \infty$ and $\alpha>0$, the left-sided Riemann-Liouville (R-L) fractional integrals of a function $f$ is defined as
$$I_{+w}^\alpha f(w)=\frac{1}{\Gamma(\alpha)}\int_{-\infty}^w (w-\eta)^{\alpha-1}
f(\eta)\mathrm{d}\eta{\color{red}.},$$
and right-sided case is defined as 
$$I_{-w}^\alpha f(w)=\frac{1}{\Gamma(\alpha)}\int_w^\infty (w-\eta)^{\alpha -1}f(\eta)\mathrm{d}\eta.$$
The order of the derivative here is $\alpha$.
\end{definition}
\begin{definition}[Riemann-Liouville Derivative, \cite{riemannliouvilleintegrallll}]
Let $f\in L^1[a,b]$ and $n=\lfloor \alpha\rfloor+1$, $\alpha>0$.  The Riemann-Liouville (RL) fractional derivative (left- and right-sided) of order $\alpha$ of the function $f$ are given as
$$D_{+w}^\alpha f(w)=\frac{1}{\Gamma(n-\alpha)}\frac{\mathrm{d}^n}{\mathrm{d}w^n}\int_{-\infty}^w (w-\eta)^{n-\alpha -1}f(\eta)\mathrm{d}\eta,$$
and 
$$D_{-w}^\alpha f(w)=\frac{(-1)^n}{\Gamma(n-\alpha)}\frac{\mathrm{d}^n}{\mathrm{d}w^n}\int_w^\infty (\eta -w)^{n-\alpha -1}f(\eta)\mathrm{d}\eta,$$
\end{definition}
respectively.
\begin{definition}[\eaa{Tempered fractional integrals, \cite{TFDE2}}]
For $1\leq p < \infty$, and for any $f \in L^p(\mathbb{R})$, the left tempered integral is defined as
\begin{equation}
\begin{split}
I_{+w}^{\alpha , \lambda} f(w) =\frac{1}{\Gamma(\alpha)}\int_{-\infty}^w(w-\eta)^{\alpha -1}e^{-\lambda(w-\eta)}f(\eta)\mathrm{d}\eta,\label{1}
\end{split}
\end{equation}
and the right tempered integral is defined as
\begin{equation}
I_{-w}^{\alpha , \lambda}f(w)=\frac{1}{\Gamma(\alpha)}\int_w^\infty (\eta - w)^{\alpha -1}e^{-\lambda(\eta - w)}f(\eta)\mathrm{d}\eta,\label{2}
\end{equation}
where $\alpha >0$ , $\lambda >0$ (the tempering parameter). 
\end{definition}
\begin{definition}[\eaa{Tempered Fractional derivative, \cite{TFDE2}}]
For any $\alpha >0$ and $\lambda >0$, the left fractional tempered derivatives of a function $f: \mathbb{R}\to \mathbb{R}$ with order $\alpha$ is defined as 
\begin{equation}
\begin{split}
D_{+w}^{\alpha , \lambda}f(w):&=e^{-\lambda w}D_{+w}^\alpha e^{\lambda w}f(w)\\
&=\frac{e^{-\lambda w}}{\Gamma(n-\alpha)}\frac{\mathrm{d}^n}{\mathrm{d}w^n}\int_{-\infty}^w (w-\eta)^{n-\alpha -1}e^{\lambda \eta}f(\eta)\mathrm{d}\eta,
\end{split}
\end{equation}
and the right is given as
\begin{equation}
 \begin{split}
D_{-w}^{\alpha , \lambda}f(w):&=e^{\lambda w}D_{-w}^\alpha e^{-\lambda w}f(w)\\
&=\frac{(-1)^ne^{\lambda x}}{\Gamma(n-\alpha)}\frac{\mathrm{d}^n}{\mathrm{d}w^n}\int_w^\infty (\eta -w)^{n-\alpha -1}e^{-\lambda \eta}f(\eta)\mathrm{d}\eta.
\end{split}
\end{equation}
\end{definition}

\begin{lemma}\cite{meer1}\label{adjoint}
Let $\lambda, \alpha >0$ and $f,g\in L^2(\mathbb{R})$. The following identity is true for the left and right-tempered fractional integrals
\begin{equation}
\bigg\langle I_+^{\alpha , \lambda}f,g \bigg\rangle=\bigg\langle f,I_-^{\alpha ,\lambda}g \bigg\rangle, \label{8}
\end{equation}
which indicates that the two integrals are adjoint to each other.
\end{lemma}

\begin{definition}[\eaa{Tempered fractional derivative spaces, \cite{TFDE2}}]
Let $\alpha, \lambda >0$. $J_{+}^{\alpha, \lambda}$ and  $J_{-}^{\alpha, \lambda}$ denote the left and right fractional tempered derivative spaces respectively which are defined as 
\begin{equation}\label{leftderivativespaces}
J_{\mp}^{\alpha, \lambda}(\mathbb{R})=\{z \in L^2 (\mathbb{R}); D{_{\mp w}^{\alpha ,\lambda}}z \in L^2 (\mathbb{R}) \},
\end{equation}
with the norm
\begin{equation}\label{JR}
\lVert z \lVert_{J_{\mp}^{\alpha, \lambda}(\mathbb{R})}=\lVert D{_{\mp w}^{\alpha ,\lambda}}z \rVert_{L^2(\mathbb{R})}.
\end{equation}
where $J_{\mp}^{\alpha, \lambda} (\mathbb{R})$ denotes the closure of $\mathbb{C}_0^{\infty}(\mathbb{R})$ with respect to $\lVert z \lVert_{J_{\mp}^{\alpha, \lambda}(\mathbb{R})}$.
\end{definition}
\begin{definition}
\cite{TFDE2} The space $H^{\alpha, \lambda}$ for $\alpha, \lambda >0$ is define as follows 
\begin{equation}\label{rightderivativespaces}
H^{\alpha, \lambda}(\mathbb{R})=\{z \in L^2 (\mathbb{R}); (\lambda^2+\omega^2)^{\alpha/2} \hat{z} \in L^2 (\mathbb{R}) \}.
\end{equation}
and the associated norm is given by
\begin{equation}
\lVert z \lVert_{H^{\alpha, \lambda}(\mathbb{R})}=\lVert (\lambda^2+\omega^2)^{\alpha/2} \hat{z} \lVert_{L^2(\mathbb{R})},
\end{equation}
\end{definition}
where $\hat{z}$ represents the Fourier transformation $z$, and the Fourier transformation variable $\omega$. In this case, the closure of $\mathbb{C}_0^{\infty}(\mathbb{R})$ with respect to $\lVert z \lVert_{H^{\alpha, \lambda}(\mathbb{R})}$ is denoted by $H^{\alpha, \lambda} (\mathbb{R})$.
%%%%%%%%%%%%%%%%%%%%%%%%%%%%%%%%%%%%%%%%%%%%%%%%%%%%%%%%%%%%%%%%%%%%%%%%(Sobolev spaces of fractional order).
\begin{definition}
\cite{TFDE2} Let $\alpha >0$. Define the spaces $H^{\alpha}$ as follows 
\begin{equation}\label{rightderivativespaces}
H^{\alpha}(\mathbb{R})=\{z \in L^2 (\mathbb{R}); (1+\omega^2)^{\alpha/2} \hat{z} \in L^2 (\mathbb{R}) \}.
\end{equation}
with the norm
\begin{equation}
\lVert z \lVert_{H^{\alpha}(\mathbb{R})} = \lVert (1+\omega^2)^{\alpha/2} \hat{z} \lVert_{L^2(\mathbb{R})}{\color{red}.}
\end{equation}
\end{definition}
\begin{lemma}
\cite{TFDE2} The following hold
\begin{equation}\label{ghghgh}
D_{\mp w}^{\alpha , \lambda} I_{\mp w}^{\alpha , \lambda}f(w) = f(w){\color{red},}
\end{equation}
for $\alpha,\,\lambda > 0$.
\end{lemma}
\begin{lemma}\label{2.11}
\cite{TFDE2} The following mappings
$$I^{\alpha,\lambda}_{\mp x}:L_2(\Omega)\rightarrow L_2(\Omega), \,\, I^{\alpha,\lambda}_{\mp x}:L_2(\Omega)\rightarrow J^{\alpha,\lambda}_{\mp}(\Omega),\,\, \text{and} \,\, D^{\alpha,\lambda}_{\mp x}:J^{\alpha,\lambda}_{\mp}(\Omega)\rightarrow L_2(\Omega)$$ are bounded linear operators for $\lambda,\, \alpha > 0$. 
\end{lemma}
\begin{lemma}\cite{TFDE2} The following hold
$$I^{\alpha,\lambda}_{\mp w} D^{\alpha,\lambda}_{\mp w} z = z$$
for every $z\in J^{\alpha,\lambda}_{\mp,0}(\Omega)$.
\end{lemma}
%%%%%%%%%%%%%%%%%%%%%%%%%%%%%%%%%%%%%%%%%%%%%%%%%%%%%%%%%%%%%%%%%%%%%%%%%%%%%%%%%%%%%%%%%%%%%%%%%%%%%%%%%%%%%%%%%%%%%%%%%%%%%%%%%%%%%%%%%%%%%%%%%%%%%%%%%%%%%%%%%%%%%%%%%%%%%%%%%%%%%%%%%%%%%%%%%%%%%%%%%%%%%%%%%%%%%%%%%%%%%%%%%%%%%%%%%%%%%%%%%%%%%%%%%%%%%%%%%%%%%%%%%%%%%%%%%%
\section{Spatial Discretization: Finite Element Method}\label{Sec:main}
The Galerkin finite element method (GFEM) is a widely utilized numerical approach in solving partial differential equations (PDEs) \eaa{in complicated domains. The method begins with a variational formulation of the PDE followed by an approximation of the weak solution in a finite dimensional function space.}\\

For completeness, we give the detailed proofs of some of the following results about tempered fractional calculus. Readers should see \cite{meer1} for further details.
\begin{theorem}\label{semi}
Let $\lambda,\alpha, \beta >0$ and $f\in L^p(\mathbb{R})$. The left-tempered fractional integrations have the following property
\begin{equation}
I_{+x}^{\alpha , \lambda}I_{+x}^{\beta , \lambda}f(x)=I_{+x}^{\alpha + \beta , \lambda}f(x). \label{7}
\end{equation}   
\end{theorem}
\begin{proof}
Applying the left fractional operator  $I_{+x}^{\alpha , \lambda}$ on $I_{+x}^{\beta , \lambda}f(x)$, we have 
\begin{align}\label{proof1}
\begin{aligned}
&I_{+x}^{\alpha , \lambda}I_{+x}^{\beta , \lambda}f(x)=\frac{1}{\Gamma(\alpha)}\int_{-\infty}^x(x-t)^{\alpha -1}e^{-\lambda(x-t)}I_{+x}^{\beta , \lambda}f(t) \mathrm{d}t\\
&=\frac{1}{\Gamma(\alpha)\Gamma(\beta)}\int_{-\infty}^x \int_{-\infty}^t (x-t)^{\alpha -1}e^{-\lambda(x-t)} (t-s)^{\beta -1}e^{-\lambda(t-s)}f(s) \mathrm{d}s  \mathrm{d}t\\
&=\frac{1}{\Gamma(\alpha)\Gamma(\beta)}\int_{-\infty}^t f(s) \int_{s}^x  (x-t)^{\alpha -1} e^{-\lambda(x-t)} (t-s)^{\beta -1}e^{-\lambda(t-s)} \mathrm{d}t  \mathrm{d}s,
\end{aligned}
\end{align}
Next, by changing variables $t = s + (x-s)r$, we have the following
\begin{equation}\label{proof2}
\begin{split}
I_{+x}^{\alpha , \lambda}I_{+x}^{\beta , \lambda}f(x)&= \frac{1}{\Gamma(\alpha)\Gamma(\beta)}\int_{-\infty}^t (x-s)^{\alpha -1} (x-s)^{\beta -1} (x-s) f(s) \\ 
&\times \int_{0}^1   e^{-\lambda(x-s)(1-r)} e^{-\lambda(x-s)r} r^ {\beta - 1}(1-r)^{\alpha - 1}\mathrm{d}r  \mathrm{d}s\\
&= \frac{1}{\Gamma(\alpha)\Gamma(\beta)}\int_{-\infty}^t (x-s)^{\alpha + \beta - 1} e^{-\lambda(x-s)}f(s) \int_{0}^1   r^ {\beta - 1}(1-r)^{\alpha - 1}\mathrm{d}r  \mathrm{d}s.
\end{split}
\end{equation}
The result follows by using Equation \ref{betaproperty}.\\
Similar procedure can be used for the right-tempered fractional integral.
\end{proof}
\begin{theorem}\label{lemma1}
\cite{TFDE1} Let $\lambda, \alpha >0$ and $f\in L^2(\mathbb{R})$. The Fourier transformations of the left and right-tempered fractional integrals are
\begin{equation}\label{LeftFouriertransformation}
\mathcal{F}\bigg[I_{\pm x}^{\alpha , \lambda}f(x)\bigg](\omega)=(\lambda \pm i \omega)^{-\alpha}\hat{f}(x).
\end{equation}
%\begin{equation}\label{LeftFouriertransformation}
%\mathcal{F}\bigg[D_{+x}^{\alpha , \lambda}f(x)\bigg](\omega)=(\lambda+i \omega)^{\alpha}\hat{f}(x),
%\end{equation}
\begin{proof}
 Recall that, for any $f \in L^2(\mathbb{R})$, its Fourier transform and inverse Fourier transform are defined by
\begin{equation}\label{Fouriertransformations1}
\mathcal{F}\bigg[f(x)\bigg](\omega)=\int_{-\infty}^\infty f(x)e^{-i \omega x}\mathrm{d}x,
\end{equation}
and
\begin{equation}\label{Fouriertransformationsinverse}
f(x)=\mathcal{F}^{-1}\bigg[\mathcal{F}[f](\omega)\bigg](x)= \frac{1}{2 \pi} \int_{-\infty}^\infty \mathcal{F}[f](\omega)e^{-i \omega x}\mathrm{d}\omega,
\end{equation}
Using the Heaviside function, $H(x)$, the the left-tempered fractional integral can be reformulated  as
\begin{equation}
 \begin{split}
I_{+x}^{\alpha , \lambda}f(x)&=\frac{1}{\Gamma(\alpha)}\int_{-\infty}^x (x-\xi)^{\alpha -1}H(\xi)e^{-\alpha (x-\xi)}f(x)\mathrm{d}\xi\\
&=(K*f)(x),
\end{split}
\end{equation}
where 
$$K(x):= x^{\alpha-1}e^{-\lambda x} \frac{H(x)}{\Gamma(\alpha)}.$$
{\bf Claim:}
\begin{equation}\label{fourierK}
\mathcal{F}[K](\omega)=(\lambda+i \omega)^{-\alpha}.
\end{equation}
\begin{proof}
Using the following identity
$$\mathcal{L}{(x^{\alpha-1};s)} = \frac{\Gamma(\alpha+1)}{s^{\alpha+1}},\ \ \alpha>0,$$
we have
\begin{eqnarray*}
\mathcal{F}[K](\omega) &=& \frac{1}{\Gamma(\alpha)}\int_{0}^\infty x^{\alpha-1} e^{-\lambda x} e^{-i \omega x}\mathrm{d}x\\
&=&\frac{1}{\Gamma(\alpha)}\int_{0}^\infty x^{\alpha-1} e^{-x(\lambda +i \omega )}\mathrm{d}x\\
&=& (\lambda +i \omega)^{-\alpha}.
\end{eqnarray*}
\end{proof}
Then, using \ref{fourierK}, we get
\begin{eqnarray*}
\mathcal{F}\bigg[I_{\pm x}^{\alpha , \lambda}f(x)\bigg](\omega) &=& \mathcal{F}[K*f](\omega)\\
&=& \mathcal{F}[K](\omega) \mathcal{F}[f](\omega) \\
&=& (\lambda \pm i \omega)^{-\alpha}\mathcal{F}[f](\omega). 
\end{eqnarray*}
\end{proof}
Hence, we obtain
\begin{equation*}\label{LeftFouriertransformation2}
\mathcal{F}\bigg[I_{\pm x}^{\alpha , \lambda}f(x)\bigg](\omega)=(\lambda \pm i \omega)^{-\alpha}\hat{f}(x).
\end{equation*}
Similarly, we have
\begin{equation}\label{fouriertransform}
\mathcal{F}\bigg[D_{\pm}^{\alpha,\lambda}u(x)\bigg](\omega)= \widehat{D_{\pm}^{{\alpha},\lambda}u(x)} = (\lambda\pm i\omega)^\alpha\hat{u}(\omega).
\end{equation}
\end{theorem}
\begin{theorem}
\cite{TFDE2} The spaces $J_{\pm}^{\alpha, \lambda}(\mathbb{R})$ and $H^{\alpha, \lambda}(\mathbb{R})$ are equal with equivalent norms.
\begin{proof} 
Using Equation \ref{fouriertransform}, we have
\begin{equation*}
\begin{split}\label{righttemperedfractionalintegrals}
\lVert u \lVert_{J_{\pm}^{\alpha, \lambda}(\mathbb{R})}&= \lVert D{_{\pm x}^{\alpha ,\lambda}}u \lVert_{L_{2}(\mathbb{R})}\\
&= \bigg(\int_\mathbb{R}D{_{+x}^{\alpha ,\lambda}}u \overline{D{_{\mp x}^{\alpha ,\lambda}}u} dx\bigg)^{1/2}\\
&=\bigg(\int_\mathbb{R}\widehat{D{_{\pm x}^{\alpha ,\lambda}}u} \overline{\widehat{D{_{\mp x}^{\alpha ,\lambda}}u}} d\omega \bigg)^{1/2}\\
&= \left(\int_\mathbb{R} (\lambda \pm i \omega )^\alpha \widehat{u(\omega)} (\lambda \mp i \omega )^\alpha \widehat{u(\omega)} d\omega \right) ^{1/2} \\
&= \left(\int_\mathbb{R} (\lambda^2 +  \omega^2 )^\alpha \lvert \widehat{u(\omega)} \rvert^2 d\omega \right) ^{1/2}\\
&= \lVert (\lambda^2 +  \omega^2 )^{\alpha/2} \widehat{u(\omega)}  \lVert_{L^2(\mathbb{R})}\\
&= \lVert u \lVert_{H^{\alpha, \lambda}(\mathbb{R})}.
\end{split}
\end{equation*}
\end{proof}
\end{theorem}

\subsection{Variational Formulation Involving Fractional Operator}
\eaa{We seek a solution } $u \in L^2(0,T:H_0^{\frac{\alpha}{2}}(\Omega))$ such that 
\begin{equation}
\langle u_t ,v \rangle-\langle \partial_{\mid x \mid}^{\alpha , \lambda}u,v \rangle=\langle f,v \rangle,\label{v1}
\end{equation}
for all $v \in L^2(0,T:H_0^{\frac{\alpha}{2}}(\Omega))$.
% where
% \begin{equation}\label{v2}
% v\in L^2(0,T;H_0^{\frac{\alpha}{2}}(\Omega)) = \{v:[0,T] \rightarrow H_0^{\frac{\alpha}{2}}(\Omega)\},%\left(\int_{\Omega}  \lVert v \rVert^2_ {H^{\frac{\alpha}{2}}(\Omega)} \mathrm{d}x \right)^{1/2}},
% \end{equation}
% associated with \eqref{MainEquation}, where $u\in L_2(0,T:H_0^\frac{\alpha}{2}(\Omega))$. 

\eaa{Applying the definition of the tempered derivative} to  \eqref{v1} we obtain
\begin{equation}
\langle u_t,v \rangle-\mathcal{C}_\alpha \bigg \langle D_{+x}^{\alpha,\lambda}+D_{-x}^{\alpha,\lambda}u-2\lambda^\alpha u,v \bigg \rangle=\langle f,v\rangle. \label{v3}
\end{equation}
\eaa{This can be further simplified to}
\begin{equation}
\langle u_t,v \rangle-\mathcal{C}_\alpha\langle D_{+x}^{\alpha,\lambda}u, v \rangle -\mathcal{C}_\alpha \langle D_{-x}^{\alpha,\lambda}u, v \rangle +2\mathcal{C}_\alpha \lambda^\alpha \langle u,v \rangle= \langle f,v \rangle. \label{v4}
\end{equation}
\eaa{By the linearity of the inner product. It remains to apply integration by parts to simplify the inner products involving the tempered derivative. While a discussion of this can be found in \cite{TFDE2}, we have added critical details which we present here for completeness.} 
% We have made a substantial contribution by presenting a thorough computation of both $\langle D_{+x}^{\alpha,\lambda}u, v \rangle$ and $\langle D_{-x}^{\alpha,\lambda}u, v \rangle$. This analysis significantly expands upon the succinct proof provided by Çelik et al. in their work .  
\eaa{
Using that fact that $D_{+x}^{\alpha , \lambda}u:=e^{-\lambda x}D_{+x}^\alpha e^{\lambda x}u$, we have
\begin{align*}
\langle D_{+x}^{\alpha,\lambda}u, v \rangle = \int_{\Omega} (D_{+x}^{\alpha , \lambda}u) v\mathrm{d}x &= \int_{\Omega} (D_{+x}^\alpha e^{\lambda x}u) e^{-\lambda x} v\mathrm{d}x.
\end{align*}
Using integrating by parts and the fact that 
$$D_{+x}^\alpha e^{\lambda x}u=D^2 I_{+x}^{2-\alpha}(e^{\lambda x}u) \quad \text{and}\quad  v \in L^2(0,T:H_0^{\frac{\alpha}{2}}(\Omega)),$$
we obtain
\begin{align}
\langle D_{+x}^{\alpha,\lambda}u, v \rangle &= \bigg[(e^{-\lambda x} v) I (D_{+x}^\alpha e^{\lambda x}u)\bigg]_{\partial \Omega} - \int_{\Omega} ID_{+x}^\alpha (e^{\lambda x}u) D (e^{-\lambda x} v) \mathrm{d}x  \nonumber \\
& = -\bigg \langle  I D_{+x}^\alpha (e^{\lambda x}u), D(e^{-\lambda x} v)\bigg \rangle \nonumber\\
& = -\bigg \langle I D^2 I_{+x}^{2-\alpha} (e^{\lambda x}u), D (e^{-\lambda x} v)\bigg \rangle \nonumber\\
& = -\bigg \langle D I_{+x}^{2-\alpha} (e^{\lambda x}u), D (e^{-\lambda x} v)\bigg \rangle \nonumber\\
& = -\bigg \langle   I_{+x}^{2-\alpha} D (e^{\lambda x}u), D (e^{-\lambda x} v)\bigg \rangle. \label{v15}
\end{align}}
Using Theorem \ref{semi} and Lemma \ref{adjoint} and the fact that $I_{+x}^{\frac{2-\alpha}{2}} D= D_{+x}^{\frac{\alpha}{2}}$, we have
\begin{equation}\label{v16}
\begin{split}
\bigg \langle D_{+x}^{\alpha,\lambda}u, v\bigg \rangle & = -\bigg \langle  I_{+x}^{2-\alpha} D (e^{\lambda x}u), D (e^{-\lambda x} v)\bigg \rangle \\
& = -\bigg \langle  I_{+x}^{\frac{2-\alpha}{2}} I_{+x}^{\frac{2-\alpha}{2}} D (e^{\lambda x}u), D (e^{-\lambda x} v)\bigg \rangle\\
& = -\bigg \langle   I_{+x}^{\frac{2-\alpha}{2}} D (e^{\lambda x}u),I_{-x}^{\frac{2-\alpha}{2}} D (e^{-\lambda x} v)\bigg \rangle\\
& =-\bigg \langle   D_{+x}^{\frac{\alpha}{2}}(e^{\lambda x}u),D_{-x}^{\frac{\alpha}{2}} (e^{-\lambda x} v)\bigg \rangle \\
& = -\bigg\langle   \mathcal{D}_{+x}^{\frac{\alpha}{2},\lambda} u,\mathcal{D}_{-x}^{\frac{\alpha}{2},\lambda}v\bigg \rangle.
\end{split}
\end{equation}
Similarly, 
\begin{align}\label{v21}
\begin{aligned}
\bigg \langle D_{-x}^{\alpha,\lambda}u, v\bigg \rangle =-\bigg \langle   D_{-x}^{\frac{\alpha}{2},\lambda} u,D_{+x}^{\frac{\alpha}{2},\lambda}v\bigg \rangle .
\end{aligned}
\end{align}
Hence, we have from \eqref{v4} for any $u,v \in H_0^{\frac{\alpha}{2}}(\Omega)$, that
\begin{equation} \label{v21}
\langle u_t,v \rangle-\mathcal{C}_\alpha\bigg \langle D_{+x}^{\frac{\alpha}{2},\lambda} u,D_{-x}^{\frac{\alpha}{2},\lambda}v\bigg \rangle -\mathcal{C}_\alpha \bigg \langle\mathcal{D}_{-x}^{\frac{\alpha}{2},\lambda} u,\mathcal{D}_{+x}^{\frac{\alpha}{2},\lambda}v\bigg \rangle  +2\mathcal{C}_\alpha \lambda^\alpha \langle u,v \rangle=\langle f,v\rangle.
\end{equation}

Then for each $t \in [0,T]$, we can write the variational form of the problem presented in \eqref{MainEquation} as
\begin{equation}\label{v22}
\langle u_t,v \rangle+R \langle u,v \rangle=\langle f , \nu \rangle,\qquad \nu\in H_0^{\frac{\alpha}{2}}(\Omega),
\end{equation}
where $R(u,v)$ (the bilinear form) is defined as
\begin{equation}\label{v24}
R(u,v)=-\mathcal{C}_\alpha\bigg \langle   D_{+x}^{\frac{\alpha}{2},\lambda} u,D_{-x}^{\frac{\alpha}{2},\lambda}v\bigg \rangle -\mathcal{C}_\alpha \bigg \langle   D_{-x}^{\frac{\alpha}{2},\lambda} u,D_{+x}^{\frac{\alpha}{2},\lambda}v\bigg  \rangle +2\mathcal{C}_\alpha \lambda^\alpha \langle u,v \rangle,
\end{equation}
and the initial condition \eqref{BC} is satisfied by $u \in H_0^{\frac{\alpha}{2}}(\Omega)$.\\

Using Equation \ref{fouriertransform} for any $\alpha>0$ and $\lambda >0$, we have
\begin{equation}\label{v25}
\begin{split}
\bigg \langle D_{+x}^{\frac{\alpha}{2},\lambda}u(\omega), D_{-x}^{\frac{\alpha}{2},\lambda}v(\omega)\bigg \rangle&= \int_{\omega} D_{+x}^{\frac{\alpha}{2},\lambda}u(\omega) D_{-x}^{\frac{\alpha}{2},\lambda}v(\omega) \mathrm{d}x \\
&= \int_{-\infty}^{\infty} \widehat{D_{+x}^{\frac{\alpha}{2},\lambda}u(\omega)} \widehat{D_{-x}^{\frac{\alpha}{2},\lambda}v(\omega)} \mathrm{d}x \\
&=\int_{-\infty}^{\infty} \widehat{D_{+x}^{\frac{\alpha}{2},\lambda}u(\omega)} \overline{\widehat{D_{-x}^{\frac{\alpha}{2},\lambda}v(\omega)}} \mathrm{d}x \\
&=\int_{-\infty}^{\infty}(\lambda + i \omega)^{\frac{\alpha}{2}}\hat{u}(\omega)\overline{(\lambda - i \omega)^{\frac{\alpha}{2}}\hat{v}(\omega)}\mathrm{d}\omega.
\end{split}
\end{equation}
Note that
\begin{equation}\label{v26}
\overline{(\lambda - i \omega)^{\frac{\alpha}{2}}}=\left\lbrace
\begin{array}{ll}
\overline{(\lambda + i \omega)^{\frac{\alpha}{2}}}e^{i \theta \alpha},& \omega \geq 0,\\
\overline{(\lambda + i \omega)^{\frac{\alpha}{2}}}e^{-i \theta \alpha},& \omega <0,
\end{array}\right.
\end{equation}
and
\begin{equation}\label{v26}
\overline{(\lambda + i \omega)^{\frac{\alpha}{2}}}=\left\lbrace
\begin{array}{ll}
\overline{(\lambda - i \omega)^{\frac{\alpha}{2}}}e^{-i \theta \alpha},& \omega \geq 0,\\
\overline{(\lambda - i \omega)^{\frac{\alpha}{2}}}e^{i \theta \alpha},& \omega <0.
\end{array}\right.
\end{equation}
Then
\begin{equation}\label{v27}
    \begin{split}
\bigg \langle D_{+x}^{\frac{\alpha}{2},\lambda}u(\omega), D_{-x}^{\frac{\alpha}{2},\lambda}v(\omega)\bigg \rangle&=\int_{-\infty}^0(\lambda + i \omega)^{\frac{\alpha}{2}}\hat{u}(\omega)\ \overline{(\lambda + i \omega)^{\frac{\alpha}{2}}\hat{v}(\omega)}e^{-i \theta \alpha}\mathrm{d}\omega \\
&+\int_0^\infty (\lambda + i \omega)^{\frac{\alpha}{2}}\hat{u}(\omega)\ \overline{(\lambda + i \omega)^{\frac{\alpha}{2}} \hat{v}(\omega)}e^{i\theta \alpha}\mathrm{d}\omega,\\
&=\int_{-\infty}^0(\lambda + i \omega)^{\frac{\alpha}{2}}\hat{u}(\omega)\ \overline{(\lambda + i \omega)^{\frac{\alpha}{2}}} \ \ \overline{ \hat{v}(\omega)}e^{-i \theta \alpha}\mathrm{d}\omega \\
&+\int_0^\infty (\lambda + i \omega)^{\frac{\alpha}{2}}\hat{u}(\omega)\ \overline{(\lambda + i \omega)^{\frac{\alpha}{2}}} \overline{\hat{v}(\omega)}e^{i\theta \alpha}\mathrm{d}\omega.\\
    \end{split}
\end{equation}
Equivalently, we have
\begin{align}\label{v29}
\begin{aligned}
\bigg \langle D_{+x}^{\frac{\alpha}{2},\lambda}u(\omega), D_{-x}^{\frac{\alpha}{2},\lambda}v(\omega)\bigg \rangle&=\int_{-\infty}^0(\lambda + i \omega)^{\frac{\alpha}{2}}\hat{u}(\omega) (\lambda - i \omega)^{\frac{\alpha}{2}} \ \ \overline{ \hat{v}(\omega)}e^{-i \theta \alpha}\mathrm{d}\omega \\
&+\int_0^\infty (\lambda + i \omega)^{\frac{\alpha}{2}}\hat{u}(\omega) (\lambda - i \omega)^{\frac{\alpha}{2}} \overline{\hat{v}(\omega)}e^{i\theta \alpha}\mathrm{d}\omega.\\
\end{aligned}
\end{align}
Hence,
\begin{equation}\label{v30}
\begin{split}
\bigg\langle D_{+x}^{\frac{\alpha}{2},\lambda}u(\omega), D_{-x}^{\frac{\alpha}{2},\lambda}v(\omega)\bigg \rangle&=\int_{-\infty}^0(\lambda^2 +  \omega^2)^{\frac{\alpha}{2}}\hat{u}(\omega) \overline{ \hat{v}(\omega)}e^{-i \theta \alpha}\mathrm{d}\omega \\
&+\int_0^\infty (\lambda^2 +  \omega^2)^{\frac{\alpha}{2}}\hat{u}(\omega) \overline{\hat{u}(\omega)}e^{i\theta \alpha}\mathrm{d}\omega.\\
\end{split}
\end{equation}
Similarly, we have 
\begin{equation}\label{v31}
\begin{split}
\bigg\langle D_{-x}^{\frac{\alpha}{2},\lambda}u(\omega), D_{+x}^{\frac{\alpha}{2},\lambda}v(\omega)\bigg\rangle&=\int_{-\infty}^0(\lambda^2 +  \omega^2)^{\frac{\alpha}{2}}\hat{u}(\omega) \overline{ \hat{v}(\omega)}e^{i \theta \alpha}\mathrm{d}\omega \\
&+\int_0^\infty (\lambda^2 +  \omega^2)^{\frac{\alpha}{2}}\hat{u}(\omega) \overline{\hat{u}(\omega)}e^{-i\theta \alpha}\mathrm{d}\omega.\\
\end{split}
\end{equation}
From \eqref{v30} and \eqref{v31}, we get
\begin{equation}\label{v32}
\begin{split}
\bigg\langle D_{+x}^{\frac{\alpha}{2},\lambda}u(\omega), D_{-x}^{\frac{\alpha}{2},\lambda}v(\omega)\bigg\rangle&+ \bigg\langle D_{-x}^{\frac{\alpha}{2},\lambda}u(\omega), D_{+x}^{\frac{\alpha}{2},\lambda}v(\omega)\bigg\rangle \\ &=\int_{-\infty}^\infty \bigg(\lambda^2 +  \omega^2\bigg)^{\frac{\alpha}{2}}\hat{u}(\omega) \overline{ \hat{v}(\omega)}\bigg(e^{i \theta \alpha}+e^{-i\theta \alpha}\bigg)\mathrm{d}\omega.
\end{split}
\end{equation}
Using 
$$(e^{i \theta \alpha}+e^{-i\theta \alpha})=2cos(\theta \alpha),$$
and Equation \ref{v32}, the relation \eqref{v24} can be written as follows 
\begin{equation}\label{v32}
R(u,v)=-2 \mathcal{C}_\alpha \int_{-\infty}^\infty \bigg(\lambda^2 +  \omega^2\bigg)^{\frac{\alpha}{2}}\hat{u}(\omega) \  \overline{\hat{v}}(\omega) \cos(\theta \alpha)\mathrm{d}\omega  +2\mathcal{C}_\alpha \lambda^\alpha \bigg\langle u(\omega),v(\omega)\bigg\rangle.
\end{equation}
\subsection{Discretization Of Variational Formulation}
Here, we discretize the variational form in \eqref{v22} to obtain a system of ordinary differential equations. Suppose we have the following representations.
\begin{equation}\label{c14}
u(x,t)=\sum_{i=1}^{N-1}u_i(t) \varphi_i (x), \ \  v(x,t)=\sum_{j=1}^{N-1}v_j(t) \varphi_j (x),\ \  f(x,t)=\sum_{i=1}^{N-1}f_i \varphi_i (x). 
\end{equation}
Substituting \eqref{c14} into \eqref{v22}, we obtain
\begin{equation}\label{c16}
\begin{split}
\bigg\langle\sum_{i=1}^{N-1}u'_i(t) \varphi_i (x),\sum_{j=1}^{N-1}v_j(t) \varphi_j (x)\bigg\rangle&+R \bigg ( \sum_{i=1}^{N-1}u_i(t) \varphi_i (x),\sum_{j=1}^{N-1}v_j(t) \varphi_j (x) \bigg )\\
&= \bigg\langle\sum_{i=1}^{N-1}f_i \varphi_i (x),\sum_{j=1}^{N-1}v_j(t) \varphi_j (x)\bigg\rangle.
    \end{split}
\end{equation}
We can write \eqref{c16} as follows
\begin{equation}\label{c17}
\begin{split}
\sum_{j=1}^{N-1}v_j(t)\bigg\langle\sum_{i=1}^{N-1}u'_i(t) \varphi_i (x), \varphi_j (x)\bigg\rangle&+ \sum_{j=1}^{N-1}v_j(t) R \bigg ( \sum_{i=1}^{N-1}u_i(t) \varphi_i (x),\varphi_j (x) \bigg) \\
&= \sum_{j=1}^{N-1}v_j(t)\bigg\langle\sum_{i=1}^{N-1}f_i \varphi_i (x), \varphi_j (x)\bigg\rangle.
    \end{split}
\end{equation}
So, we have
\begin{equation}\label{c18}
\begin{split}
\bigg\langle\sum_{i=1}^{N-1}u'_i(t) \varphi_i (x), \varphi_j (x)\bigg\rangle &+ R\bigg(\sum_{i=1}^{N-1}u_i(t) \varphi_i (x),\varphi_j (x)\bigg)\\
&= \bigg\langle\sum_{i=1}^{N-1}f_i \varphi_i (x), \varphi_j (x)\bigg\rangle, \ \ j=1,...,N-1.
\end{split}
\end{equation}
Then
\begin{equation}\label{c19}
\begin{split}
\sum_{i=1}^{N-1}u'_i(t) \bigg\langle\varphi_i (x), \varphi_j (x)\bigg\rangle &+ \sum_{i=1}^{N-1}u_i(t) R\bigg(\varphi_i (x),\varphi_j (x)\bigg) \\ &
= \sum_{i=1}^{N-1}f_i \bigg\langle\varphi_i (x), \varphi_j (x)\bigg\rangle.
\end{split}
\end{equation}
Consider $(N-1)\times (N-1)$ matrices $P$ and $G$ having entries $$p_{i,j}=\bigg\langle\varphi_j(x),\varphi_i(x)\bigg\rangle,\quad \text{and}\quad g_{i,j}=R \bigg (\varphi_j(x),\varphi_i(x)\bigg).$$
respectively.\\

Also, 
$$U=\bigg[u_1(t), \cdots , u_{N-1}(t)\bigg]^T, \ \ U'=\bigg[u'_1(t), \cdots , u'_{N-1}(t)\bigg]^T,$$ and $F=\bigg[f_1(t), \cdots , f_{N-1}(t)\bigg]^T$, we get, from Equation \eqref{c19}
\begin{equation}\label{c20}
PU'+ GU = PF.
\end{equation}
Taking $ B = P^{-1}G$, we obtained the system of ordinary differential equations as follows
\begin{equation*}
\begin{split}
U'+ BU = F.
\end{split}
\end{equation*}
\section{Time Discretization: ETD-RDP and ETD-CN}
In this section, our goal is to present overview of the time discretization methods adopted to solve the system of ordinary differential equations derived in the previous section
\begin{equation}\label{c21}
U'+ BU = F(t,U).
\end{equation}
The focus is on exponential time differencing with real distinct poles rational approximation to the exponential. For completeness, we also give the Crank-Nicolson method. We should note here that both ETD-RDP and CN schemes are second order accurate schemes.

\subsection{ETD-RDP Scheme}
We give brief steps in constructing the ETD-RDP scheme solving system of ordinary differential equation in \eqref{c21}.  Using a variation of the constants formula on the time interval $[t_m,t_{m+1}]$, we have
\begin{equation}
U(t_{m+1}) = e^{-\tau B}U(t_m) +\int_{t_m}^{t_{m+1}} e^{-B(t_{m+1}-s)}F(s,U(s))ds.
\end{equation}
Taking $s = t_m + z \tau $ with $t_m = m\tau$, $z \in [0, 1]$ and $\tau$ is the time-step, the following recurrence formula is obtained:
\begin{equation}\label{int1}
U(t_{m+1}) = e^{-\tau B}U(t_m) + \tau \int_{0}^{1}e^{-\tau B(1-z)}F(t_m + z\tau,U(t_m+z \tau))dz.
\end{equation}
When dealing with the discretizations of the integral part in \eqref{int1}, many approximations result into different schemes with different properties such as stability and convergence orders. In the case of ETD-RDP scheme, we seek a second-order scheme in which a linear approximation of the nonlinear function is used as follows:
\begin{equation}\label{ffbf}
F(t_m + z \tau, U(t_m + z \tau)) \approx F(t_m, U(t_m)) + z  \tau \frac{F(t_{m+1}, U(t_{m+1})) - F(t_m, U(t_m))}{\tau}.
\end{equation}
Using the above approximation in Equation \eqref{int1}, we obtain a semi-discretized scheme with $U(t_m)\approx U_m$
\begin{eqnarray*}
U_{m+1} &=& e^{-B\tau}U_m + B^{-1}\left[I - e^{-B\tau}\right]F(t_m,U_m) \\
&& + \frac{B^{-2}}{\tau}\left[\tau B - I + e^{-B \tau}\right]\bigg[ F(t_{m+1},U_{m+1})-F(t_m,U_m)\bigg].
\end{eqnarray*}
By employing the constant approximation in \eqref{int1} and integrating, first order accuracy approximation is obtained as
\begin{eqnarray}
U_{m+1}^* = e^{-\tau B}U_m + B^{-1}\left[I - e^{-\tau B}\right]F(t_m,U_m),
\end{eqnarray}
The final semi-discretized scheme is then obtained as:
\begin{eqnarray}
U_{m+1}^* &=& e^{-B \tau}U_n + B^{-1}\left(I - e^{-B \tau}\right)F(t_m,U_m)\\
U_{m+1} &=& e^{-B\tau}U_m + B^{-1}\left( I - e^{-B \tau} \right) F(t_m, U_m) \nonumber \\
&& + \frac{B^{-2}}{\tau}\left[\tau B - I+e^{-B \tau}\right]\bigg[F(t_{n+1},U^*_{m+1})-F(t_m,U_m)\bigg]. 
\end{eqnarray}
The primary computational incentive for approximating the exponential function lies in its efficiency and practicality. This leads us to introduce the RDP rational function as an approximation method \cite{asanRDPP1,asanRDPP5}. This approximation technique is L-acceptable, ensuring a reliable and satisfactory outcome  \cite{asanRDPP1,asanRDPP5}. Consider the following rational function approximation of the exponential function.
\begin{equation*}
r(z) = \frac{1+\frac{5}{12}z}{(1-\frac{1}{4}z)(1-\frac{1}{3}z)} =\frac{9}{1-\frac{1}{3}z} -\frac{8}{1-\frac{1}{4}z}\approx e^z.
\end{equation*}
The function $r(z)$  is second order accurate. Assume $r(z)$ is defined and bounded on the spectrum of $\tau B$, then
\begin{eqnarray}\label{rdp}
e^{-\tau B}\approx r(-\tau B) &=& \left[I-\frac{5}{12}\tau B\right]\left[I+\frac{1}{4}\tau B\right]^{-1} \left[I+\frac{1}{3}\tau B\right]^{-1}\nonumber\\
&=& 9\left[I+\frac{1}{3}\tau B\right]^{-1} - 8\left[I+\frac{1}{4}\tau B\right]^{-1}.
\end{eqnarray}
One major advantage of the above approximation is the fact that it can be decomposed into partial fraction, allowing basic algebra to simplify the scheme and also permit parallel implementation of the final scheme. Consider the following partial fraction decomposition
\begin{eqnarray}\label{appr1}
\left[I +\frac{B\tau }{4}\right]^{-1}\left[I+\frac{B\tau }{3}\right]^{-1} &\approx & 4\left[I+\frac{1}{3}\tau B\right]^{-1} - 3\left[I+\frac{1}{4}\tau B\right]^{-1} \\ \label{appr2}
\left[I+\frac{1}{6}\tau B\right]\left[I+\frac{1}{4}\tau B\right]^{-1} \left[I+\frac{1}{3}\tau B\right]^{-1} &\approx & 2\left[I+\frac{1}{3}\tau B\right]^{-1} - \left[I+\frac{1}{4}\tau B\right]^{-1}.
\end{eqnarray}
Combining Equations \eqref{rdp}, \eqref{appr1} and \eqref{appr2}, we have the final ETD-RDP scheme as
\begin{eqnarray}
W_{m+1}^* &=& \left[I + B \tau\right]^{-1}\left[W_n + \tau F(t_m,W_m)\right],\\
W_{m+1} &=& \left[I + \frac{1}{3}B \tau\right]^{-1} \bigg(9W_m + 2\tau F(t_m,W_m) + \tau F(t_{m+1},W^*_{m+1})\bigg] \nonumber \\
&& +  \left[I + \frac{1}{4}B \tau\right]^{-1} \bigg[-8W_m - \frac{3 \tau}{2} F(t_m,W_m) - \frac{\tau}{2} F(t_{m+1},W^*_{m+1})\bigg]. 
\end{eqnarray}
\subsection{Crank-Nicolson Scheme}   
We present overview of Crank-Nicolson scheme here which is used for comparing purposes in the implementation section. Consider the same system of ordinary differential equations in \eqref{c14} and the following approximations:
$$\bar{\partial}U_m=\frac{U_m - U_{m-1}}{\tau}, \ \ \tilde{U}_m=\frac{U_m + U_{m-1}}{2},  \ \ F^{m-\frac{1}{2}}=\frac{1}{2}\bigg[F(U_m,t_m)+F(U_{m+1},t_{m+1})\bigg].$$
Hence, the Crank-Nicolson final scheme is given as
\begin{equation}\label{c25}
\bigg[I + \frac{1}{2}B\tau\bigg] W_m  =  \bigg[I - \frac{1}{2}B\tau\bigg] W_{m-1} + F^{m-\frac{1}{2}},
\end{equation}
where Newton's method is used to handle the nonlinear function.
%%%%%%%%%%%%%%%%%%%%%%%%%%%%%%%%%%%%%%%%%%%%%%%%%%%%%%%%%%%%%%%%%%%%%%%%%%%%%%%%%%%%%%%%%%%%%%%%%%%%%%%%%%%%%%%%%%%%%%%%%%%%%%%%%%%%%%%%%%%%%%%%%%%%%%%%%%%%%%%%%%%%%%%%%%%%%%%%%%%%%%%%%%%%%%%%%%%%%%%%%%%%%%%%%%%%%%%%%%%%%%%%%%%%%%%%%%%%%%%%%%%%%%%%%%%%%%%%%%%%%%%%%%%%%%%%%%
\section{Numerical Examples and Discussion}\label{sec:numerics}
This section is focused on the implementation of the proposed scheme ETD-RDP together with Finite Element Method (ETD-RDP-FEM) for solving fractional reaction diffusion equations. The impact of the fractional order in the problems considered on the efficacy of the scheme is further investigated. The order of convergence is calculated using the formula \cite{convergenc1}
$$p \approx \frac{\log\left(\frac{E(\tau)}{E(\tau/2)}\right)}{\log(2)},  \ \tau = \frac{T}{M}
,$$
where $M$ denotes the number of discretization points in the temporal direction and error $E(\tau)$ is given by:
\begin{equation*}
E(\tau) = \frac{\lVert u - \tilde{u} \rVert_{\infty}}{\lVert u \rVert_{\infty}},
\end{equation*}
where $u$ represents the exact solution and $\tilde{u}$ represents the approximate solution.
%%%%%%%%%%%%%%%%%%%%%%%%%%%%%%%%%%%%%%%%%%%%%%%%%%%%%%%%%%%%%%%%%%%%%%%%%%%%%%%%%%%%%%%%%%%%%%%%%%%%%%%%%%%%%%%%%%%%%%%%%%%%%%%%%%%%%%%%%%%%%%%%%%%%%%%%%%%%%%%%%%%%%%%%%%%%%%%%%%%%%%%%%%%%%%%%%%%%%%%%%%%%%%%%%%%%%%%%%%%%%%%%%%%%%%%%%%%%%%%%%%%%%%%%%%%%%%%%%%%%%%%%%%%%%%%%%%
\subsection{Numerical Examples}
\begin{exm}[{\bf Non-homogeneous Linear Model}]\label{Example1}$\,$\\
Consider the following non-homogeneous linear model of Riesz-tempered fraactional order derivative:
\begin{equation}
\begin{split}
u_t(x,t)-\partial_{\mid x \mid}^{\alpha ,\lambda}u(x,t)&=f(x,t),  \ (x, t) \in [0, 1] \times (0, 1],\\
u(0,t)&=u(1,t)=0, \\
u(x,0)&=g(x),
\label{meample1}
    \end{split}
\end{equation}
with the exact solution 
$$u(x, t) = -2 \lambda^{6 - \alpha} \cos \left(\frac{\pi \alpha}{2}\right) (-\alpha) e^{-t} x^3 (1 - x)^3,$$
and function $f(x,t)$ is given by:
\begin{eqnarray}\label{source_term}\nonumber
f(x,t) &=& 2\lambda^{6-\alpha}\cos\left(\frac{\pi\alpha}{2}\right)\Gamma(-\alpha)e^{-t}x^3(1-x)^3 \\ \nonumber
&& + e^{-t} \left( g_1(x) + \lambda \left[g_2(x) + g_3(x) + g_4(x) + g_5(x) + g_6(x) + g_7(x)\right] + g_8(x) + g_9(x)\right).\\
\end{eqnarray}

where
\begin{eqnarray*}
g_1(x) &=& 3\lambda^5(x-1)^2x^2(2x-1)\gamma(1-\alpha, x\lambda)\\
g_2(x) &=& -\lambda^5x^6\gamma(-\alpha, x\lambda) + 3\lambda^5x^5\gamma(-\alpha, x\lambda) - 3\lambda^5x^4\gamma(-\alpha, x\lambda) - 15\lambda^3x^4\gamma(2-\alpha, \lambda-x\lambda)\\
&& + \lambda^5x^3\gamma(-\alpha, x\lambda) - \lambda^5(x-1)^3x^3\gamma(-\alpha, \lambda-x\lambda) + 30\lambda^3x^3\gamma(2-\alpha, \lambda-x\lambda)\\
g_3(x) &=& + 20\lambda^2x^3\gamma(3-\alpha, x\lambda) - 20\lambda^2x^3\gamma(3-\alpha, \lambda-x\lambda) - 30\lambda^2x^2\gamma(3-\alpha, x\lambda)\\
&& - 3\lambda^4(x-1)^2(2x-1)x^2\gamma(1-\alpha, \lambda-x\lambda) - 18\lambda^3x^2\gamma(2-\alpha, \lambda-x\lambda)\\
&& + 30\lambda^2x^2\gamma(3-\alpha, \lambda-x\lambda)\\
g_5(x) &=& - 15\lambda x^2\gamma(4-\alpha, x\lambda) - 3\lambda^3(x-1)(5(x-1)x+1)x\gamma(2-\alpha, x\lambda) \\
&& - 15\lambda x^2\gamma(4-\alpha, \lambda-x\lambda) \\
g_6(x) &=& + 3\lambda^3x\gamma(2-\alpha, \lambda-x\lambda) + 12\lambda^2x\gamma(3-\alpha, x\lambda) - 12\lambda^2x\gamma(3-\alpha, \lambda-x\lambda)\\
&& - \lambda^2\gamma(3-\alpha, x\lambda) + \lambda^2\gamma(3-\alpha, \lambda-x\lambda) + 15\lambda x\gamma(4-\alpha, x\lambda)\\
&& + 15\lambda x\gamma(4-\alpha, \lambda-x\lambda)\\
g_7(x) &=& + 6x\gamma(5-\alpha, x\lambda) - 6x\gamma(5-\alpha, \lambda-x\lambda) - 3\lambda\gamma(4-\alpha, x\lambda) - 3\lambda\gamma(4-\alpha, \lambda-x\lambda) \\
&& - 3\gamma(5-\alpha, x\lambda) + 3\gamma(5-\alpha, \lambda-x\lambda)\\
g_8(x) &=& +2\Gamma(2-\alpha) \bigg( \alpha^4 - 14\alpha^3 + 71\alpha^2 - 154\alpha + 3(\alpha - 3)(\alpha - 2)\lambda^2(5(x - 1)x + 1)\\
&&+ 3\lambda^4(x - 1)x(5(x - 1)x + 1) + 120 \bigg)\\
g_9(x) &=& - \gamma(6-\alpha, x\lambda) - \gamma(6-\alpha, \lambda-x\lambda),
\end{eqnarray*}
and the incomplete Gamma function, denoted as $\gamma(z, x)$, is defined as:
\begin{equation*}
\gamma(z, x) = \int_0^x t^{z-1}e^{-t} dt.
\end{equation*}
As a proof of concept, we first introduce a test case example where the exact solution is known. Although this is a linear model, the fractional order in space makes it a difficult problem to solve due to nonlocal property effect.

\begin{table}[H]
\centering
\caption{Example \eqref{Example1}: Time rate of convergence and $L_\infty$-norm of ETD-RDP-FEM and CN-FEM with different values of $\alpha$ at fixed $t = 1$.}
\begin{tabular}{ccccccc}
\toprule
\multicolumn{1}{c}{} & \multicolumn{1}{c}{} & \multicolumn{1}{c}{} & \multicolumn{2}{c}{\textbf{Norm}} & \multicolumn{2}{c}{\textbf{Order}} \\
\cmidrule(rl){4-5} \cmidrule(rl){6-7}
$\alpha$ & $h$ & $\tau$ & {CN-FEM} & {ETD-RDP-FEM} & {CN-FEM} & {ETD-RDP-FEM} \\
\midrule
1.2 & 1/4 & 1/4 & $8.7000 \times 10^{-3}$ & $8.4801 \times 10^{-3}$ & 1.95623 & 1.91356 \\
& 1/8 & 1/8 & $2.2000 \times 10^{-3}$ & $2.3000 \times 10^{-3}$ & 1.99015 & 2.07254 \\
& 1/16 & 1/16 & $5.5377 \times 10^{-4}$ & $5.4680 \times 10^{-4}$ & 2.08542 & 2.05254 \\
& 1/32 & 1/32 & $1.3048 \times 10^{-4}$ & $1.3181 \times 10^{-4}$ & 2.04652 & 2.03512 \\
\hline
1.4 & 1/4 & 1/4 & $7.4000 \times 10^{-4}$ & $6.1000 \times 10^{-3}$ & 2.12561 & 2.08985 \\
& 1/8 & 1/8 & $1.7000 \times 10^{-3}$ & $1.4000 \times 10^{-3}$ & 2.11021 & 2.06416 \\
& 1/16 & 1/16 & $3.9374 \times 10^{-4}$ & $8.1890 \times 10^{-5}$ & 2.08141 & 2.01553 \\
& 1/32 & 1/32 & $9.3034 \times 10^{-5}$ & $2.0253 \times 10^{-5}$ & 2.05750 & 2.01035 \\
\hline
1.6 & 1/4 & 1/4 & $6.3352 \times 10^{-4}$ & $1.2544 \times 10^{-4}$ & 2.09511 & 2.08252 \\
& 1/8 & 1/8 & $1.4828 \times 10^{-4}$ & $2.9617 \times 10^{-5}$ & 2.08856 & 2.05658 \\
& 1/16 & 1/16 & $3.4863 \times 10^{-5}$ & $7.1195 \times 10^{-6}$ & 2.06758 & 2.03521 \\
& 1/32 & 1/32 & $8.3169 \times 10^{-6}$ & $1.7370 \times 10^{-6}$ & 2.04651 & 2.01042 \\
\hline
1.8 & 1/4 & 1/4 & $1.1100 \times 10^{-4}$ & $1.6543 \times 10^{-5}$ & 2.06253 & 2.04523 \\
& 1/8 & 1/8 & $2.6573 \times 10^{-5}$ & $4.0081 \times 10^{-6}$ & 2.05568 & 2.03435 \\
& 1/16 & 1/16 & $6.3917 \times 10^{-6}$ & $9.7845 \times 10^{-7}$ & 2.03658 & 2.00525 \\
& 1/32 & 1/32 & $1.5579 \times 10^{-6}$ & $2.4372 \times 10^{-7}$ & 2.02546 & 2.00358 \\
\bottomrule
\end{tabular}\label{TableExample1}
\end{table}

% \begin{figure}[H]\label{FE2}
% \centering
% \includegraphics[height=7cm,width=7cm]{MConvergence1.6.png}\includegraphics[height=7cm,width=7cm]{MEfficiency1.6.png}\\
% \caption{Example \eqref{Example1}: Convergence and Efficiency Plots of ETD-RDP vs CN Schemes.}\label{FigureExample1}
% \end{figure}
% \end{exm}

\begin{figure}[H]
\centering
\includegraphics[height=6cm,width=8cm]{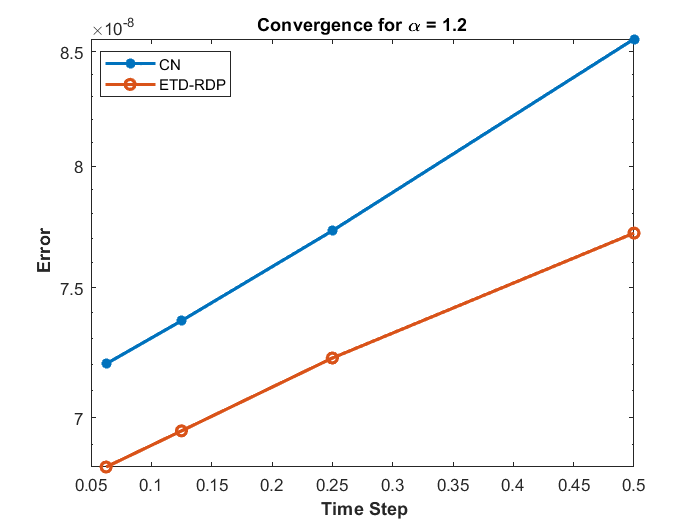}\includegraphics[height=6cm,width=8cm]{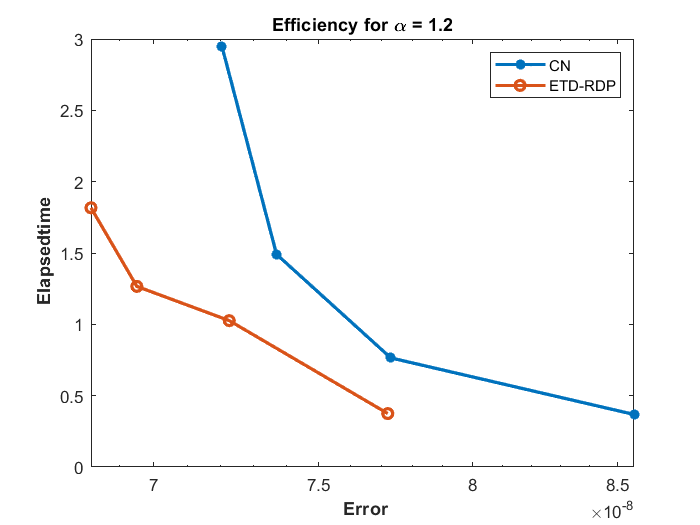}\\
\includegraphics[height=6cm,width=8cm]{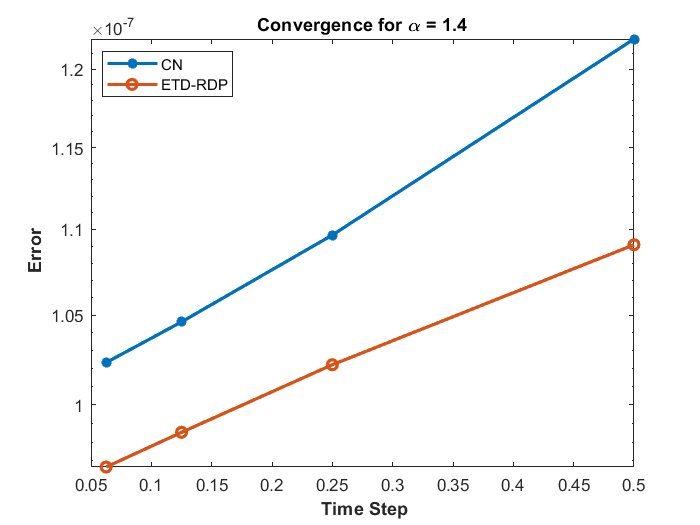}\includegraphics[height=6cm,width=8cm]{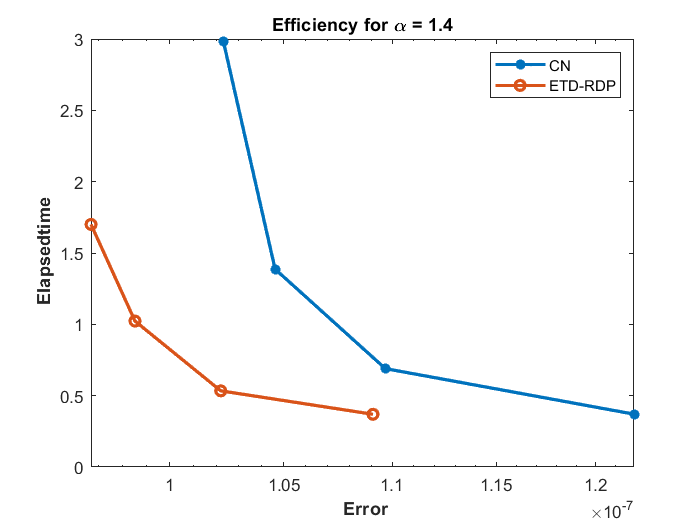}\\
\includegraphics[height=6cm,width=8cm]{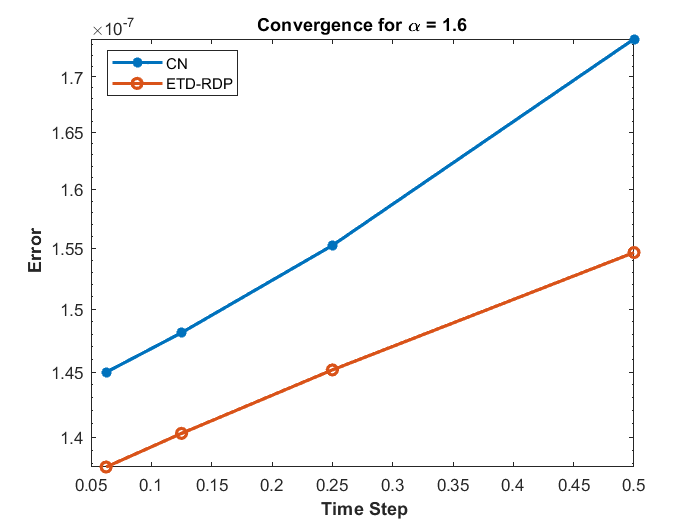}\includegraphics[height=6cm,width=8cm]{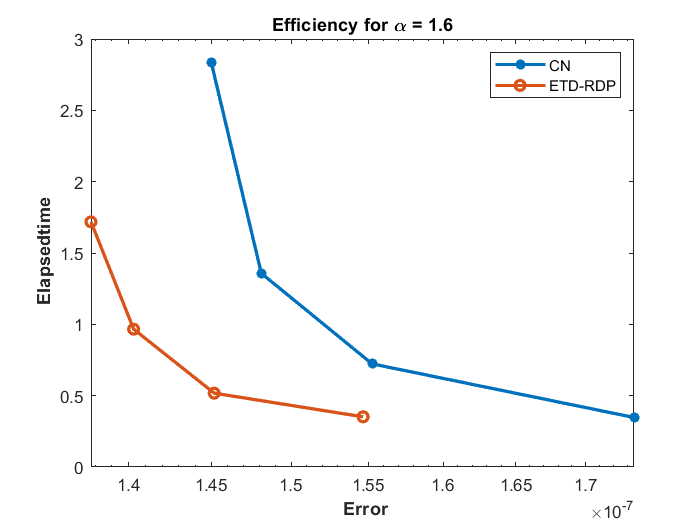}\\
\includegraphics[height=6cm,width=8cm]{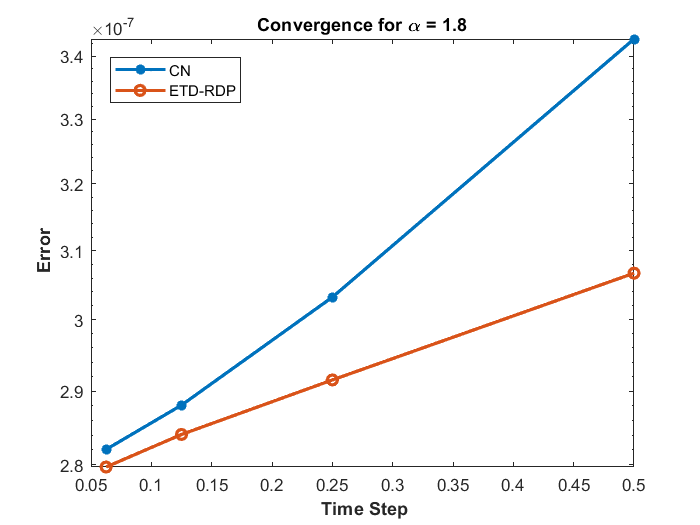}\includegraphics[height=6cm,width=8cm]{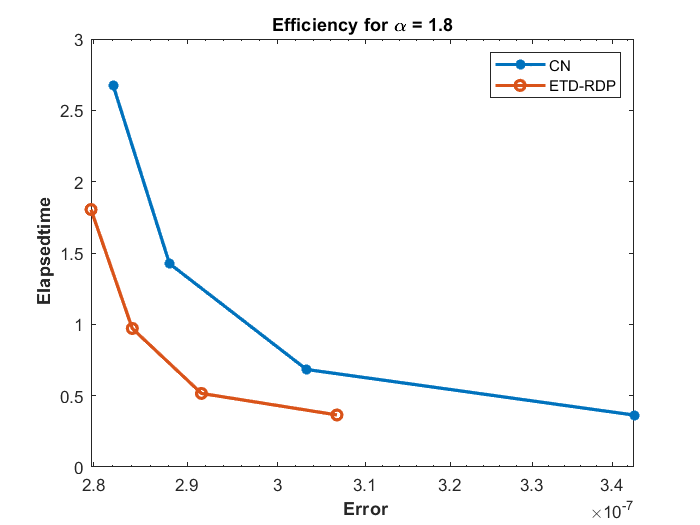}
\caption{Example \eqref{Example1}: Convergence and Efficiency Plots of ETD-RDP-FEM vs CN-FEM Schemes for Different Values of $\alpha$.}\label{FE1}
\end{figure}

\begin{figure}[H]
\centering
\includegraphics[height=6.5cm,width=8.2cm]{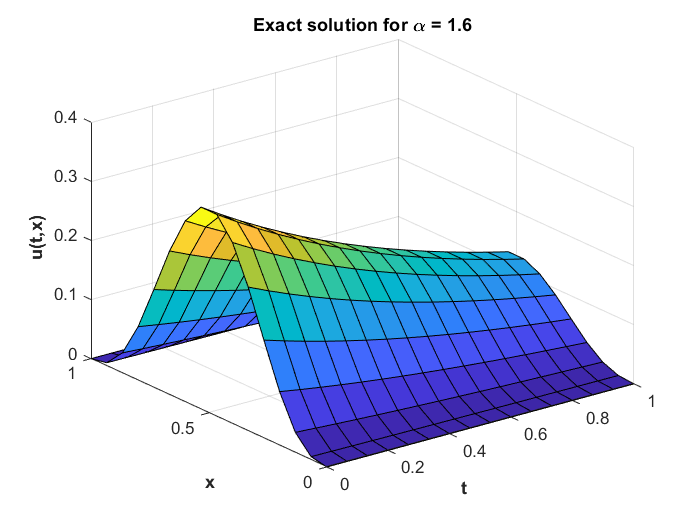}\includegraphics[height=6.5cm,width=8.2cm]{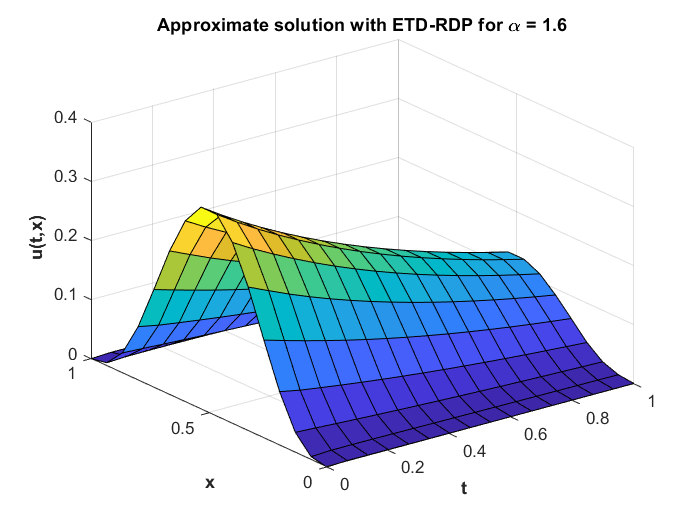}
\includegraphics[height=6.5cm,width=8.2cm]{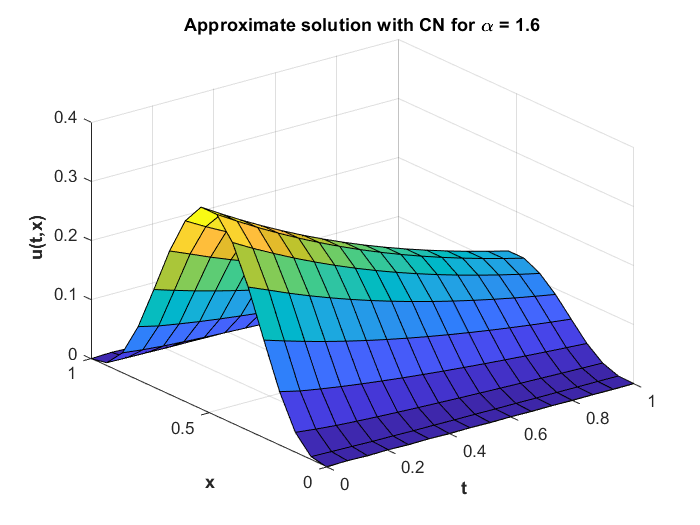}
\caption{Example \ref{Example1}: Exact Solution vs ETD-RDP-FEM vs CN-FEM Schemes for $\alpha = 1.6$.}\label{exact1}
\end{figure}
It is not surprising that both ETD-RDP-FEM and CN-FEM methods perform well in this case as this is a linear model. However, the efficiency of CN-FEM method declines as the fractional order $\alpha$ increases as can be seen in Table \ref{TableExample1} and Figures \ref{FE1} and \ref{exact1}. Overall, a second order convergence is attained by both methods as well.

% \begin{figure}[H]\label{FE2}
% \centering
% \includegraphics[height=7cm,width=7cm]{MConvergence1.6.png}\includegraphics[height=7cm,width=7cm]{MEfficiency1.6.png}\\
% \includegraphics[height=7cm,width=7cm]{MConvergence1.8.png}\includegraphics[height=7cm,width=7cm]{MEfficiency1.8.png}
% \caption{\small Graphs of convergence and efficiency of ETD-RDP and CN methods for Experiment \eqref{Example1}}.\label{FE2}
% \end{figure}
\end{exm}
%%%%%%%%%%%%%%%%%%%%%%%%%%%%%%%%%%%%%%%%%%%%%%%%%%%%%%%%%%%%%%%%%%%%%%%%%%%%%%%%%%%%%%%%%%%%%%%%%%%%%%%%%%%%%%%%%%%%%%%%%%%%%%%%%%%%%%%%%
\begin{exm}[{\bf Nonlinear Model}]\label{Example2}$\,$\\
Here, we investigate the performance of the proposed ETD-RDP-FEM scheme for the following nonlinear Riesz-tempered fractional reaction diffusion equation:
\begin{equation}
\begin{split}
u_t(x,t)-\partial_{\mid x \mid}^{\alpha ,\lambda}u(x,t)&=f(x,t,u)\\
u(0,t)&=u(1,t)=0, \\
u(x,0)&=g(x),
\end{split}
\end{equation}
where  
$$u(x, t) = x^2(1 - x)^2e^{-t}$$
and 
\begin{eqnarray*}
f(x,t,u) &=& u^2 - x^2(1 - x)^2e^{-t} - x^4(1 - x)^4e^{-2t}\\
&& + \frac{e^{-t}}{2cos(\alpha \pi /2)}\bigg[ e^{-\lambda x} H_{k,m,\lambda}(x) +  e^{-\lambda(x+1)}\bar{H}_{k,m,\lambda}(x) - 2\lambda ^ \alpha x^2(1 - x)^2 \bigg],
\end{eqnarray*}
with
\begin{eqnarray*}
H_{k,m,\lambda}(x) &=& \sum_{k=0}^\infty \sum_{m=2}^4 \frac{\lambda^k \Gamma (k +m+ 1)A_m}{\Gamma (k + 1) \Gamma (m + 1 - \alpha))}x^{k+m},\\
\bar{H}_{k,m,\lambda}(x)&=&\sum_{k=0}^\infty \sum_{m=2}^4 \frac{\lambda^k \Gamma (k +m+ 1)A_m}{\Gamma (k + 1) \Gamma (m + 1 - \alpha))}(1-x)^{k + m -\alpha},
\end{eqnarray*}
and $A_2 = 1$, $A_3 = -2$, and $A_4 = 1$.\\
The nonlinearity stems from the polynomial $u^2$ term in the function $f$. The numerical solution from both schemes agree very well with the exact solution (see Figure~\ref{FE22NN}). The second order convergence of the ETD-RDP-FEM scheme is evident in Table~\ref{TableExample2NN} for all values of $\alpha$. The scheme is also more accurate and significantly faster than the existing CN-FEM scheme (see Figure~\ref{FE2NN}).
\noindent

% \begin{table}[H]
% \centering
% \caption{Comparing Order of convergence and $L_\infty$ norm for ETD-RDP and CN for $h = 1/512$ for Non-Linear Model.}
% \begin{tabular}{cccccc}
% \toprule
% \multicolumn{1}{c}{} & \multicolumn{1}{c}{} & \multicolumn{2}{c}{\textbf{Error}} & \multicolumn{2}{c}{\textbf{Order}} \\
% \cmidrule(rl){3-4} \cmidrule(rl){5-6}
% $\alpha$ & $\tau$ & {CN} & {ETD-RDP} & {CN} & {ETD-RDP} \\
% \midrule
% 1.2 & 1/4 & 2.6491 \times 10^{-3} & 5.0404 \times 10^{-4} & 1.920398 & 2.132421 \\
% & 1/8 & 6.9984 \times 10^{-4} & 1.1496 \times 10^{-4} & 1.919881 & 2.052436 \\
% & 1/16 & $1.8495 \times 10^{-4}$ & $2.7714 \times 10^{-5}$ & 1.96965 & 2.035465 \\
% & 1/32 & 4.7221 \times 10^{-5} & 6.7603 \times 10^{-6} & 2.12556 & 2.024797 \\ \\
% 1.8 & 1/4 & 1.1200 \times 10^{-4} & 1.1153 \times 10^{-4} &1.99142 & 2.065871 \\
%  & 1/8 & 2.8167 \times 10^{-5} & 2.6638 \times 10^{-5} & 2.22130 & 2.02320 \\
% & 1/16 & 6.0404 \times 10^{-6} & 2.6598 \times 10^{-6} & 2.13117 & 2.02156 \\
% & 1/32 & 1.3288 \times 10^{-6} & 1.0243 \times 10^{-6} & 2.18452 & 2.03460 \\ \label{TableExample2NN}
% \bottomrule
% \end{tabular}
% \end{table}
\begin{table}[H]
\centering
\caption{Example \eqref{Example2}: Time rate of convergence and $L_\infty$-norm of ETD-RDP-FEM and CN-FEM with different values of $\alpha$ at fixed $t = 1$.}.
\begin{tabular}{ccccccc}
\toprule
\multicolumn{1}{c}{} & \multicolumn{1}{c}{} & \multicolumn{1}{c}{} & \multicolumn{2}{c}{\textbf{Norm}} & \multicolumn{2}{c}{\textbf{Order}} \\
\cmidrule(rl){4-5} \cmidrule(rl){6-7}
$\alpha$ & $h$ & $\tau$ & {CN-FEM} & {ETD-RDP-FEM} & {CN-FEM} & {ETD-RDP-FEM} \\
\midrule
1.2 & 1/4 & 1/4 & $7.5000 \times 10^{-3}$ & $4.5000 \times 10^{-3}$ & 2.18640 & 2.09653 \\
& 1/8 & 1/8 & $1.6000 \times 10^{-3}$ & $1.1000 \times 10^{-3}$ & 2.09254 & 2.07329 \\
& 1/16 & 1/16 & $3.7515 \times 10^{-4}$ & $2.6138 \times 10^{-4}$ & 2.08117 & 2.04752 \\
& 1/32 & 1/32 & $8.8656 \times 10^{-5}$ & $6.3228 \times 10^{-5}$ & 2.05712 & 2.02367 \\
\hline
1.4 & 1/4 & 1/4 & $5.5000 \times 10^{-3}$ & $6.1000 \times 10^{-3}$ & 2.12561 & 2.08561 \\
& 1/8 & 1/8 & $1.7000 \times 10^{-3}$ & $1.30000 \times 10^{-3}$ & 2.11021 & 2.07142 \\
& 1/16 & 1/16 & $3.9374 \times 10^{-4}$ & $3.0930 \times 10^{-4}$ & 2.08141 & 2.04653 \\
& 1/32 & 1/32 & $9.3034 \times 10^{-5}$ & $7.4871 \times 10^{-5}$ & 2.0575 & 2.02142 \\
\hline
1.6 & 1/4 & 1/4 & $3.1000 \times 10^{-3}$ & $9.0000 \times 10^{-4}$ & 2.08589 & 2.07125 \\
& 1/8 & 1/8 & $7.3021 \times 10^{-4}$ & $2.1416 \times 10^{-4}$ & 2.07661 & 2.05415 \\
& 1/16 & 1/16 & $1.7311 \times 10^{-4}$ & $5.1568 \times 10^{-5}$ & 2.06910 & 2.03536 \\
& 1/32 & 1/32 & $4.1254 \times 10^{-5}$ & $1.2580 \times 10^{-5}$ & 2.04113 & 2.01250 \\
\hline
1.8 & 1/4 & 1/4 & $5.5631 \times 10^{-4}$ & $1.6328 \times 10^{-4}$ & 2.06862 & 2.05004 \\
& 1/8 & 1/8 & $1.3262 \times 10^{-4}$ & $3.9428 \times 10^{-5}$ & 2.05108 & 2.04129 \\
& 1/16 & 1/16 & $3.2002 \times 10^{-5}$ & $9.5789 \times 10^{-6}$ & 2.05056 & 2.02258 \\
& 1/32 & 1/32 & $7.7250 \times 10^{-6}$ & $2.3575 \times 10^{-6}$ & 2.03401 & 2.01961 \\
\bottomrule
\end{tabular}\label{TableExample2NN}
\end{table}

\begin{figure}[H]
\centering
\includegraphics[height=6cm,width=8cm]{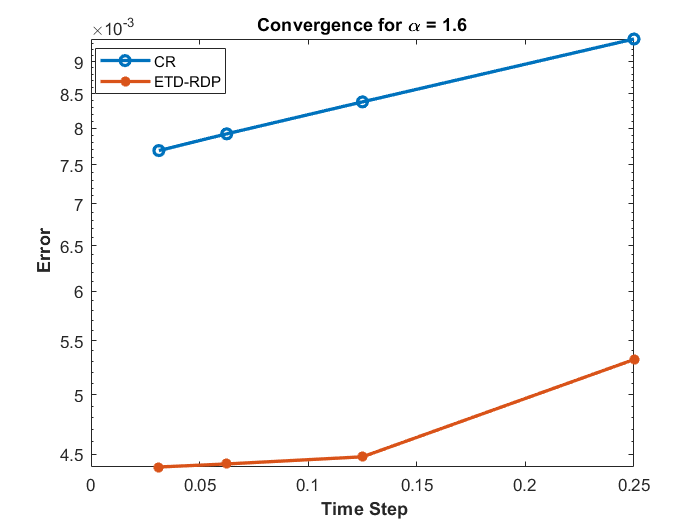}\includegraphics[height=6cm,width=8cm]{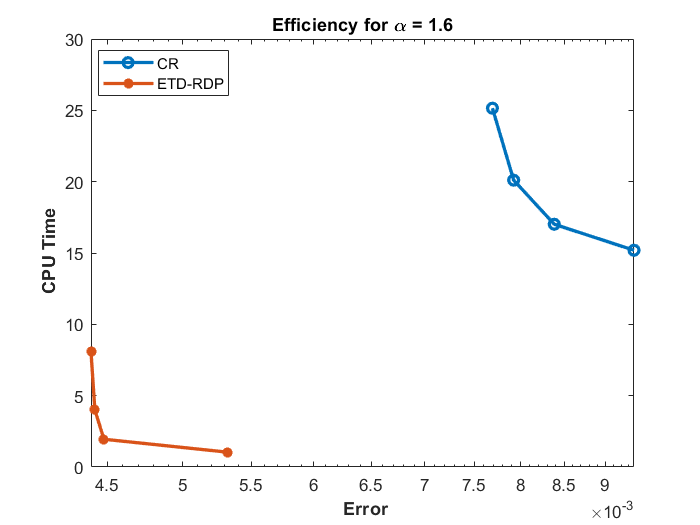}
\caption{Example \eqref{Example2}: Convergence and Efficiency Plots of ETD-RDP-FEM vs CN-FEM Schemes}
\label{FE2NN}
\end{figure}

\begin{figure}[H]
\centering
\includegraphics[height=6.5cm,width=8.2cm]{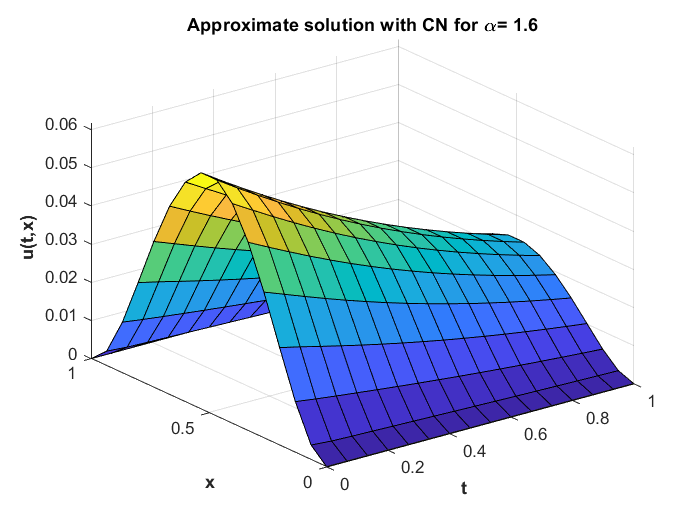}\includegraphics[height=6.5cm,width=8.2cm]{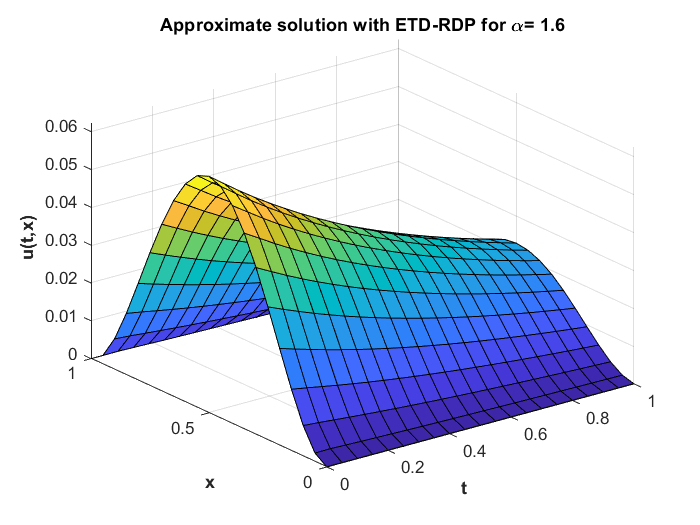}
\includegraphics[height=6.5cm,width=8.2cm]{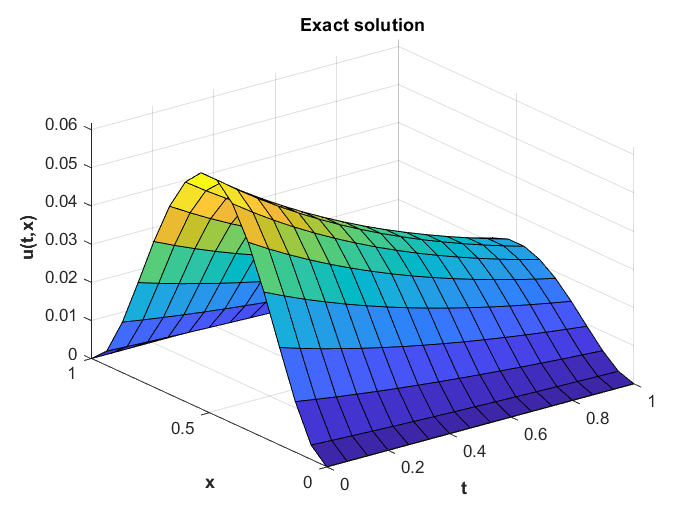}
\caption{Example \ref{Example2}: Exact Solution vs ETD-RDP-FEM vs CN-FEM Schemes}
\label{FE22NN}
\end{figure}
\end{exm}
%%%%%%%%%%%%%%%%%%%%%%%%%%%%%%%%%%%%%%%%%%%%%%%%%%%%%%%%%%%%%%%%%%%%%%%%%%%%%%%%%%%%%%%%%%%%%%%%%%%%%%%%%%%%%%%%%%%%%%%%%%%%%%%%%%%%%%%%%
\begin{exm}[{\bf Model with Nonsmooth Data}]\label{Example3}$\,$\\
Consider the non-smooth initial-boundary value problem in the following tempered
fractional diffusion equation:
\begin{equation}\label{meample3}
\begin{split}
u_t(x,t)-\partial_{\mid x \mid}^{\alpha ,\lambda}u(x,t)&=f(x,t),\\
u(0,t)&=u(1,t)=0, \\
\end{split}
\end{equation}
with the initial condition
\begin{equation}\label{MABC1}
u(x,0)=\left\lbrace
\begin{array}{ll}
0 \ \ \quad 0 \leq x < 0.25,\\
1 \ \ \quad 0.25 \leq x < 0.75, 
\\
0 \ \ \quad 0.75 \leq x \leq 1.
\end{array}\right.
\end{equation}
\begin{figure}[H]
\centering
\includegraphics[height=5.5cm,width=9cm]{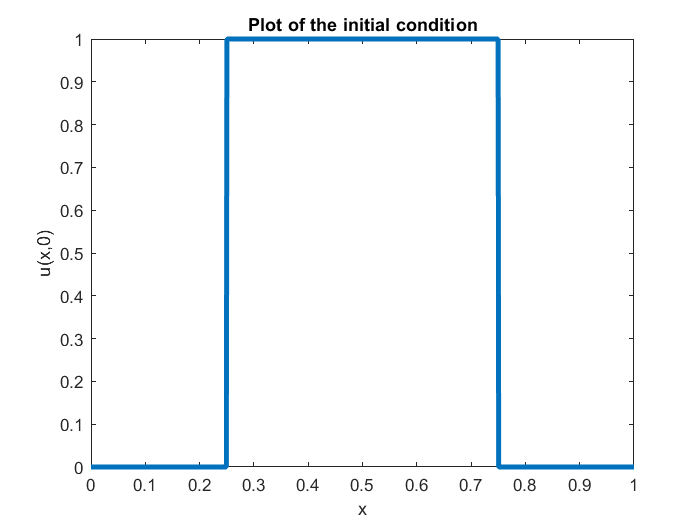}
\caption{\small Plot of the piecewise function $u(x,0)$.}
\end{figure}
The source term, $f(x,t)$ here is the same as defined in Equation \ref{source_term}. Note that Example \ref{Example3} has a nonsmooth initial data which could lead to spurious oscillation. It is worthwhile to understand how the both methods considered here handle a problem with such initial condition. We use a finer grid ($h = 1/512$) for our reference solution in this case since there is no known exact solution.
\begin{table}[H]
\centering
\caption{Comparing order of convergence and $L_\infty$ norm for ETD-RDP-FEM and CN-FEM for $h = 1/512$ for Experiment \eqref{Example3}}.
\begin{tabular}{cccccc}
\toprule
\multicolumn{1}{c}{} & \multicolumn{1}{c}{} & \multicolumn{2}{c}{\textbf{Norm}} & \multicolumn{2}{c}{\textbf{Order}} \\
\cmidrule(rl){3-4} \cmidrule(rl){5-6}
$\alpha$ & $\tau$ & {CN-FEM} & {ETD-RDP-FEM} & {CN-FEM} & {ETD-RDP-FEM} \\
\midrule
1.2 & 1/4 & 6.20970$\times 10^{-5}$ & 2.85963$\times 10^{-5}$ & 2.01558 & 1.83061 \\
& 1/8 & 1.53575$\times 10^{-5}$ & 8.0398$\times 10^{-6}$ & 2.02613 & 2.09858 \\
& 1/16 & 4.7408$\times 10^{-6}$ & 1.8772$\times 10^{-6}$ & 2.08905 & 2.01992 \\
& 1/32 & 1.1143$\times 10^{-6}$ & 5.713$\times 10^{-7}$ & 1.9765 & 2.05117 \\
\hline
1.4 & 1/4 & 1.21690$\times 10^{-4}$ & 1.15635$\times 10^{-4}$ & 1.73515 & 2.01515 \\
& 1/8 & 3.65528$\times 10^{-5}$ & 2.8607$\times 10^{-5}$ & 1.99108 & 2.20616 \\
& 1/16 & 9.1949$\times 10^{-6}$ & 7.1194$\times 10^{-6}$ & 1.89055 & 2.00655 \\
& 1/32 & 2.8074$\times 10^{-6}$ & 1.7534$\times 10^{-6}$ & 2.10271 & 2.02159 \\
\hline
1.6 & 1/4 & 8.59655$\times 10^{-4}$ & 2.15255$\times 10^{-4}$ & 1.98175 & 2.02479 \\
& 1/8 & 2.1765$\times 10^{-4}$& 5.2897$\times 10^{-5}$ & 2.08775 & 2.04210 \\
& 1/16 & 5.1202$\times 10^{-5}$ & 1.2903$\times 10^{-5}$ & 2.09112 & 2.03548 \\
& 1/32 & 1.2017$\times 10^{-5}$ & 3.2220$\times 10^{-6}$ & 2.21559 & 2.00170 \\
\hline
1.8 & 1/4 & 3.45616$\times 10^{-4}$ & 1.01124$\times 10^{-4}$ &1.99570 & 2.11327 \\
& 1/8 & 8.6662$\times 10^{-5}$ & 2.3372$\times 10^{-5}$ & 1.93860 & 2.02320 \\
& 1/16 & 2.2608$\times 10^{-5}$ & 5.7985$\times 10^{-6}$ & 2.13117 & 2.00015 \\
& 1/32 & 5.1608$\times 10^{-6}$ & 1.4495$\times 10^{-6}$ & 2.10124 & 2.00100 \\
\bottomrule
\end{tabular}\label{TableExample3}
\end{table}

% \begin{table}[H]
% \centering
% \caption{Comparing Order of convergence and $L_\infty$ norm for ETD-RDP and CN for $h = 1/512$ for Nonsmooth Model}.
% \begin{tabular}{cccccc}
% \toprule
% \multicolumn{1}{c}{} & \multicolumn{1}{c}{} & \multicolumn{2}{c}{\textbf{Error}} & \multicolumn{2}{c}{\textbf{Order}} \\
% \cmidrule(rl){3-4} \cmidrule(rl){5-6}
% $\alpha$ & $\tau$ & {CN} & {ETD-RDP} & {CN} & {ETD-RDP} \\
% \midrule
% 1.2 & 1/4 & 6.2097$\times 10^{-5}$ & 2.85963$\times 10^{-5}$ & 2.01558 & 1.83061 \\
% & 1/8 & 1.5357$\times 10^{-5}$ & 8.0398$\times 10^{-6}$ & 2.02613 & 2.09858 \\
% & 1/16 & 4.7408$\times 10^{-6}$ & 1.8772$\times 10^{-6}$ & 2.08905 & 2.01992 \\
% & 1/32 & 1.1143$\times 10^{-6}$ & 5.713$\times 10^{-7}$ & 1.9765 & 2.05117 \\ \\
% 1.8 & 1/4 & 3.4561$\times 10^{-4}$ & 1.0112$\times 10^{-4}$ &1.99570 & 2.11327 \\
% & 1/8 & 8.6662$\times 10^{-5}$ & 2.3372$\times 10^{-5}$ & 1.93860 & 2.02320 \\
% & 1/16 & 2.2608$\times 10^{-5}$ & 5.7985$\times 10^{-6}$ & 2.13117 & 2.00015 \\
% & 1/32 & 5.1608$\times 10^{-6}$ & 1.4495$\times 10^{-6}$ & 2.10124 & 2.00100 \\ \label{TableExampleMM}
% \bottomrule
% \end{tabular}
% \end{table}

\begin{figure}[H]
\centering
\includegraphics[height=7cm,width=7cm]{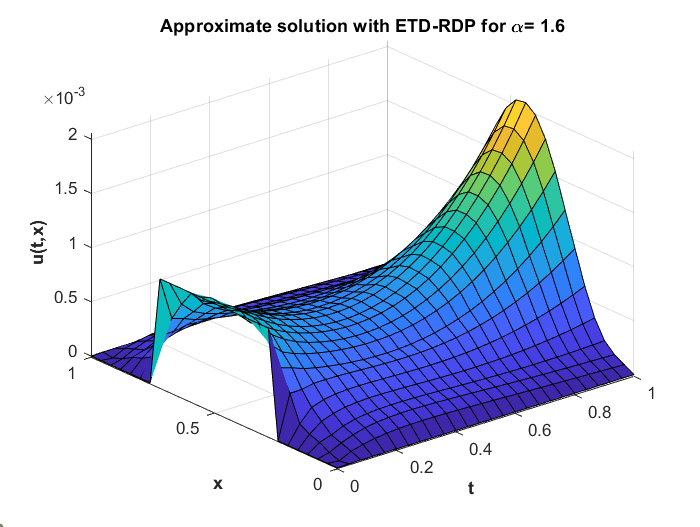}\includegraphics[height=7cm,width=7cm]{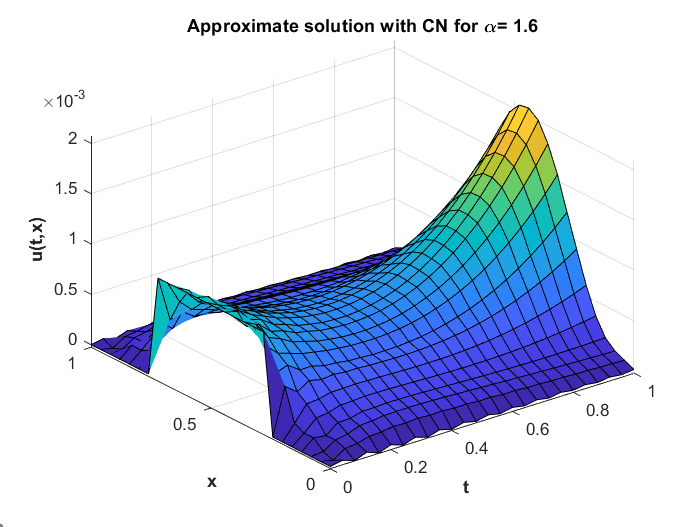}
\caption{Example \ref{Example3}: ETD-RDP-FEM vs CN-FEM Schemes}
\end{figure}

\begin{figure}[H]
\centering
\includegraphics[height=7cm,width=7cm]{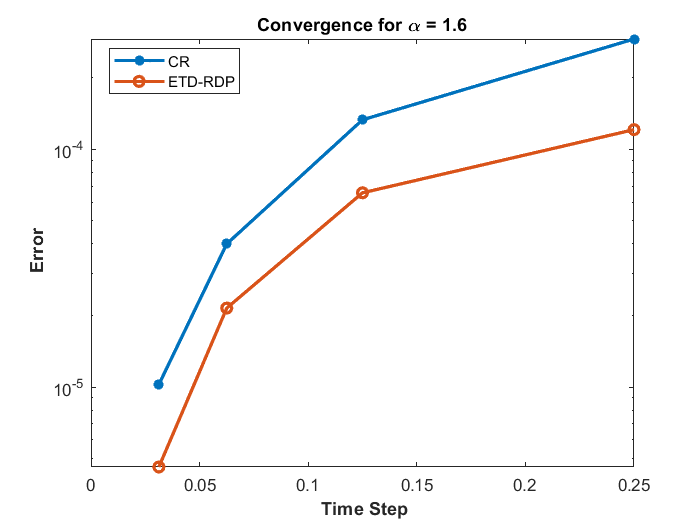}\includegraphics[height=7cm,width=7cm]{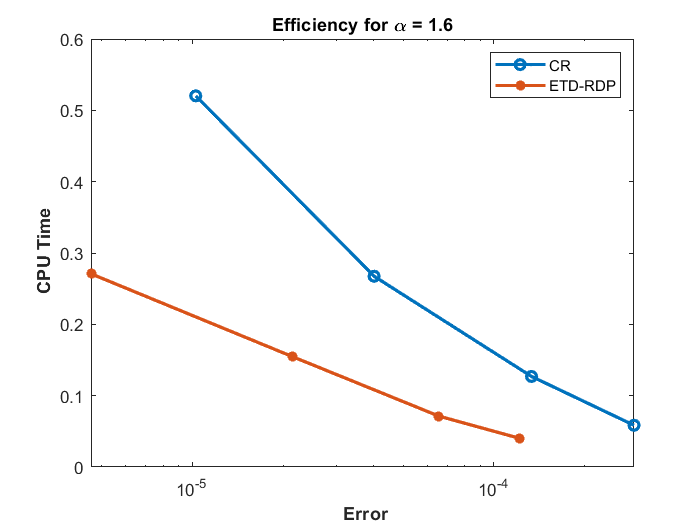}
\caption{Example \eqref{Example3}: Convergence and Efficiency Plots of ETD-RDP-FEM vs CN-FEM Schemes.}\label{FE3MM}
\end{figure}
We observe from the above results in Example \ref{Example3} also that second-order convergence is achieved by both ETD-RDP-FEM and CN-FEM, see \ref{TableExample3}. However, ETD-RDP-FEM outperformed CN-FEM across different values of $\alpha$ as well here. This observation agrees with what has been established in the literature regarding CN method.

%\begin{figure}[H]
%\centering
%\includegraphics[height=6.5cm,width=6.5cm]{MMappETD.png}\includegraphics[height=6.5cm,width=6.5cm]{MMexETD.png}
%\includegraphics[height=6.5cm,width=6.5cm]{MMerrETD.png}
%\caption{\small Graphs of the approximate, exact solutions and absolute error using ETD-RDP method for the nonsmooth experiment. }
%\end{figure}

\end{exm}
%%%%%%%%%%%%%%%%%%%%%%%%%%%%%%%%%%%%%%%%%%%%%%%%%%%%%%%%%%%%%%%%%%%%%%%%%%%%%%%%%%%%%%%%%%%%%%%%%%%%%%%%%%%%%%%%%%%%%%%%%%%%%%%%%%%%%%%%%%%%%%%%%%%%%%%%%%%%%%%%%%%%%%%%%%%%%%%%%%%%%%%%%%%%%%%%%%%%%%%%%%%%%%%%%%%%%%%%%%%%%%%%%%%%%%%%%%%%%%%%%%%%%%%%%%%%%%%%%%%%%%%%%%%%%%%%%%
\subsection{Discussion}
The focus here is to discuss the results obtained from the above Examples \eqref{Example1}-\eqref{Example3} in the previous subsection especially the impact of the fractional order $\alpha$ on the performance of numerical schemes. In particular, we discuss the effect of $\alpha$ on CPU time and accuracy. This is very important to explore in order to gain insight into the impact of the non-integer order on the considered problems in the previous Section.
\subsubsection{Effect of Fractional Order $\alpha$ on CPU Time}
It is widely known that the efficiency of a method could be greatly affected depending on the values of fractional order $alpha$. The impact is usually noticed in the CPU time. We have performed experiments across different values of $alpha$ here to understand the effect on the efficiency of the methods of interest here.
% \begin{table}[H]
% \centering
% \caption{Comparing ETD-RDP and CN in terms of Error and CPU time for $\alpha = 1.6$ for experiment \eqref{Example1}.}
% \begin{tabular}{ccccc}
% \toprule
% \multicolumn{1}{c}{} & \multicolumn{2}{c}{\textbf{Error}} & \multicolumn{2}{c}{\textbf{CPU time}} \\
% \cmidrule(rl){2-3} \cmidrule(rl){4-5}
% $\tau$ & {ETD-RDP} & {CN} & {ETD-RDP} & {CN} \\
% \midrule
% $1/4$ & $1.3776\times10^{-7}$ & $1.4498\times10^{-7}$ & 0.516978  & 0.559070 \\
% $1/8$ & $1.4028\times10^{-7}$ & $1.4810\times10^{-7}$ & 0.975105 & 1.154735 \\
% $1/16$ & $1.4516\times10^{-7}$ & $1.5523\times10^{-7}$ & 1.486886 & 0.985207 \\
% $1/32$ & $1.5464\times10^{-7}$ & $1.7351\times10^{-7}$ & 2.846924 & 1.771452 \\
% \bottomrule
% \end{tabular}\label{TableExample1CPU}
% \end{table}

\begin{table}[H]
\centering
\caption{Comparing Efficiency of ETD-RDP-FEM and CN-FEM. The time step $\tau = 1/32$ and mesh size $h = 1/512$}
\begin{tabular}{ccccc}
\toprule[1.5pt]
\multicolumn{2}{c}{} & \multicolumn{3}{c}{\textbf{CPU time} (in secs)} \\
\cmidrule(rl){1-5} 
$\alpha$ & Methods & Linear & Nonlinear & Nonsmooth\\
\toprule[1.5pt]
1.2  & {ETD-RDP-FEM} & 1.687328 & 8.415730 & 0.285885 \\
     & {CN-FEM}      & 2.783019 & 17.379510 & 0.506686\\
\toprule[1.5pt]
1.4  & {ETD-RDP-FEM} & 1.791548 & 8.204377 & 0.319528 \\
     & {CN-FEM}      & 2.862972 & 20.911838 & 0.532703 \\
\toprule[1.5pt]
1.6  & {ETD-RDP-FEM} & 1.703123 & 8.018364 & 0.270826 \\
     & {CN-FEM}      & 2.881956 & 24.915706 & 0.520299 \\
\toprule[1.5pt]
1.8  & {ETD-RDP-FEM} & 1.833103 & 8.377576 & 0.266322 \\
     & {CN-FEM}      & 3.047225 & 33.042088 & 0.491660\\
\toprule[1.5pt]
\end{tabular}\label{TableExample1OrderCPU}
\end{table}

\begin{figure}[H]
\minipage{0.33\textwidth}
\includegraphics[width=\linewidth]{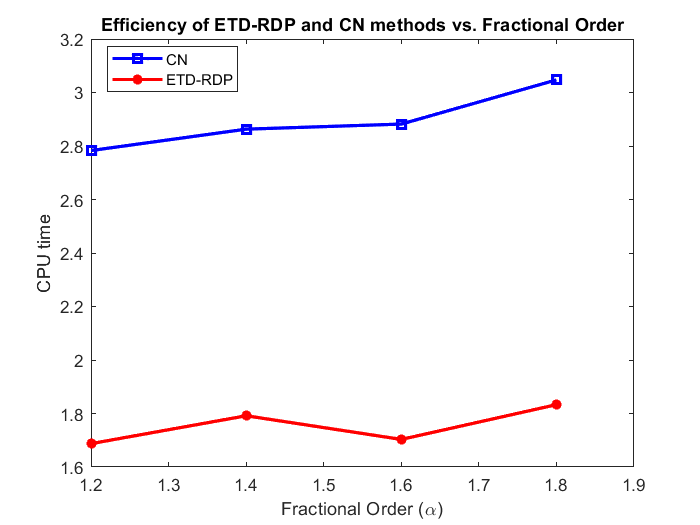}
\caption*{$(a)$ Linear Model}\label{FigureExample1OrderCPU1}
\endminipage\hfill
\minipage{0.33\textwidth}
\includegraphics[width=\linewidth]{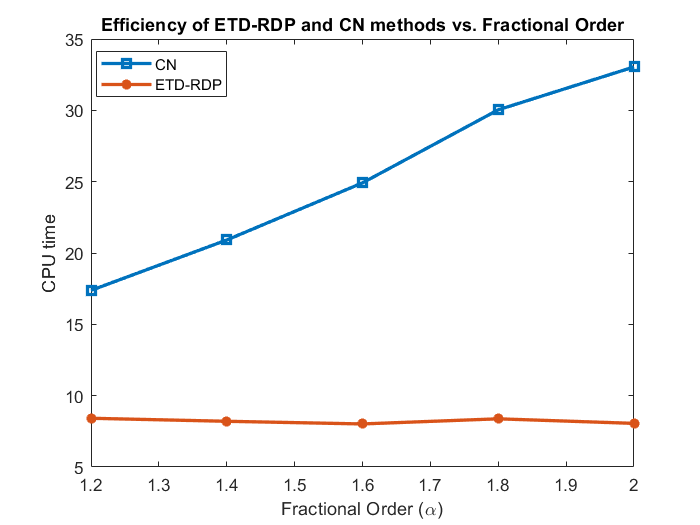}
\caption*{$(b)$ Nonlinear Model}\label{FigureExample0OrderCPUNN}
\endminipage
\hfill
\minipage{0.33\textwidth}
\includegraphics[width=\linewidth]{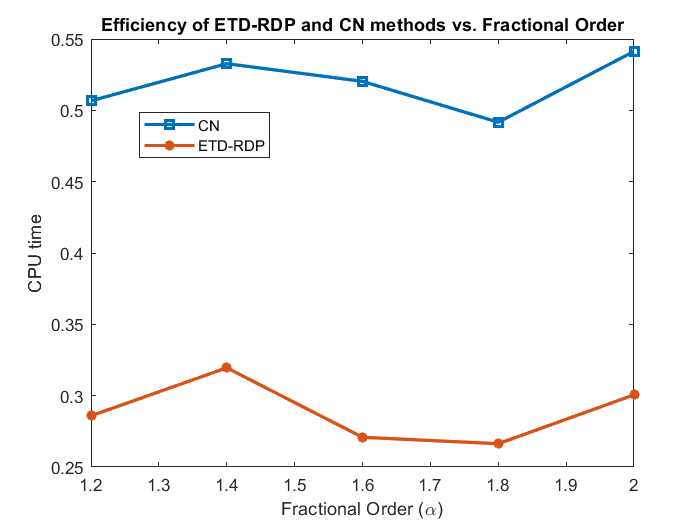}
\caption*{$(c)$ Nonsmooth Model}\label{FigureExample0OrderCPUNN}
\endminipage
\caption{Computational Efficiency Plots for Examples \eqref{Example1} - \eqref{Example3}}\label{efficiencyplots}
\end{figure}

\begin{table}[H]
\centering
\caption{Percentage Improvement of ETD-RDP-FEM over CN-FEM}
\begin{tabular}{cccc}
\toprule
\multicolumn{1}{c}{$\alpha$} & \multicolumn{1}{c}{Nonsmooth} & \multicolumn{1}{c}{Nonlinear} & \multicolumn{1}{c}{Linear} \\
\midrule
1.2 & 43.57\% & 51.57\% & 39.42\% \\
1.4 & 40.02\% & 60.78\% & 37.41\% \\
1.6 & 47.97\% & 67.81\% & 40.86\% \\
1.8 & 45.84\% & 72.10\% & 39.86\% \\
2 & 44.47\% & 75.61\% & - \\
\bottomrule
\end{tabular}\label{PercentageImprovementt}
\end{table}

\begin{figure}[H]
\centering
\includegraphics[height=8cm,width=12cm]{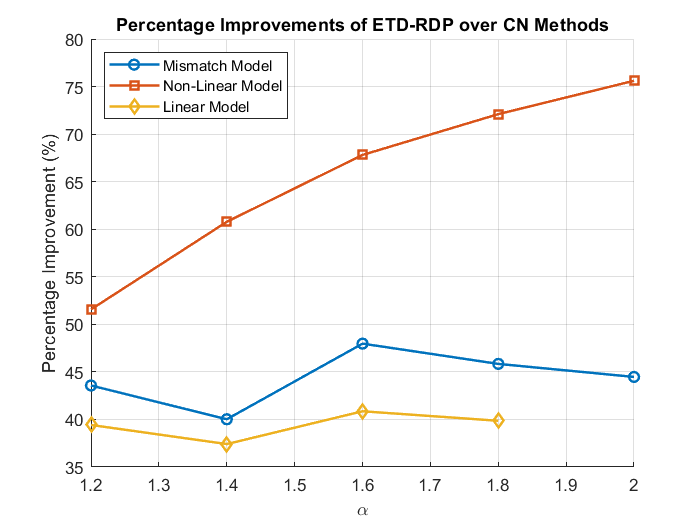}
\caption{\small Plot of the percentage improvements in CPU time of ETD-RDP-FEM over CN-FEM across different Models and fractional order $\alpha$.}\label{PercentageImprovementf}
\end{figure}

The comparison of the ETD-RDP-FEM with CN-FEM is computationally more efficient for different values of fractional orders in all the Examples considered as can be seen in Table \ref{TableExample1OrderCPU}. It is notable to see that the efficiency of both the ETD-RDP-FEM and the CN-FEM in the linear and nonsmooth models is approximately linear, see Figures \ref{efficiencyplots} $(a)$ and $(c)$. However, in the case of nonlinear model, CN-FEM is much slower than ETD-RDP-FEM as shown in Figure \ref{efficiencyplots}$(b)$. This is expected as ETD-RDP-FEM requires no nonlinear solver unlike CN-FEM. Newton method is used to handle the nonlinear term here with predefined tolerance $10^{-6}$ which is fixed across different values of $alpha$. We further demonstrate the superiority of efficiency of ETD-RDP-FEM over CN-FEM in Table \ref{PercentageImprovementt} and Figure \ref{PercentageImprovementf}. The difference is efficiency is calculate in percentage across different values of $a\alpha$. On average, ETD-RDP-FEM shows notable improvement over CN-FEM of approximately $47.97\%$, $67.81\%$, and $40.86\%$ in the linear, nonlinear and nonsmooth cases, respectively.

\subsubsection{Effect of Fractional Order $\alpha$ on Accuracy}
Our numerical experiments for all test problems show that ETD-RDP-FEM produces more accurate solutions than the existing CN-FEM scheme. Comparing the errors for each scheme across different values of the fractional order $\alpha$, we observe from Table~\ref{TableExample1OrderError} that numerical errors tends to increase as $\alpha$ approaches 1 and decreases as $\alpha$ approaches 2 for the linear and nonlinear model (Figure~\ref{accuracyplots} a,b). This observation is consistent with the theoretical results and performance of many numerical schemes for solving fractional differential equations \cite{yusuf1,wang,musti}. However, for the nonsmooth problem the error increases with increasing alpha up to $\alpha=1.6$ and then decreases. This is an interesting observation, worthy of further exploration (Figure~\ref{accuracyplots} c).
\begin{table}[H]
\centering
\caption{Comparing the accuracy of ETD-RDP-FEM and CN-FEM across different fractional order $\alpha$. The time step $\tau = 1/32$ and mesh size $h = 1/512$}
\begin{tabular}{ccccc}
\toprule[1.5pt]
\multicolumn{2}{c}{} & \multicolumn{3}{c}{\textbf{$L_{\infty}$ Error}} \\
\cmidrule(rl){1-5} 
$\alpha$ & Methods & Linear & Nonlinear & Nonsmooth\\
\toprule[1.5pt]
1.2  & {ETD-RDP-FEM} & $1.0640\times 10^{-5}$ & $6.7603\times 10^{-6}$& $5.7130\times 10^{-7}$ \\
     & {CN-FEM}      & $3.1163\times 10^{-5}$ & $4.7221\times 10^{-5}$ & $1.1143\times 10^{-6}$\\
\toprule[1.5pt]
1.4  & {ETD-RDP-FEM} & $6.5507\times 10^{-6}$ & $3.2633\times 10^{-6}$ & $1.7534\times 10^{-6}$ \\
     & {CN-FEM}      & $1.4364\times 10^{-5}$ & $4.5148\times 10^{-5}$ & $2.8074\times 10^{-6}$ \\
\toprule[1.5pt]
1.6  & {ETD-RDP-FEM} & $1.5818\times 10^{-6}$ & $1.4703\times 10^{-6}$ & $3.2200\times 10^{-6}$ \\
     & {CN-FEM}      & $2.4090\times 10^{-6}$ & $4.4593\times 10^{-6}$ & $1.2017\times 10^{-5}$ \\
\toprule[1.5pt]
1.8  & {ETD-RDP-FEM} & $1.0600\times 10^{-6}$ & $1.0243\times 10^{-6}$ & $1.4495\times 10^{-6}$ \\
     & {CN-FEM}      & $1.7028\times 10^{-6}$ & $1.3288\times 10^{-6}$ & $5.1608\times 10^{-6}$\\
\toprule[1.5pt]
\end{tabular}
\label{TableExample1OrderError}
\end{table}

\begin{figure}[H]
\minipage{0.33\textwidth}
\includegraphics[width=\linewidth]{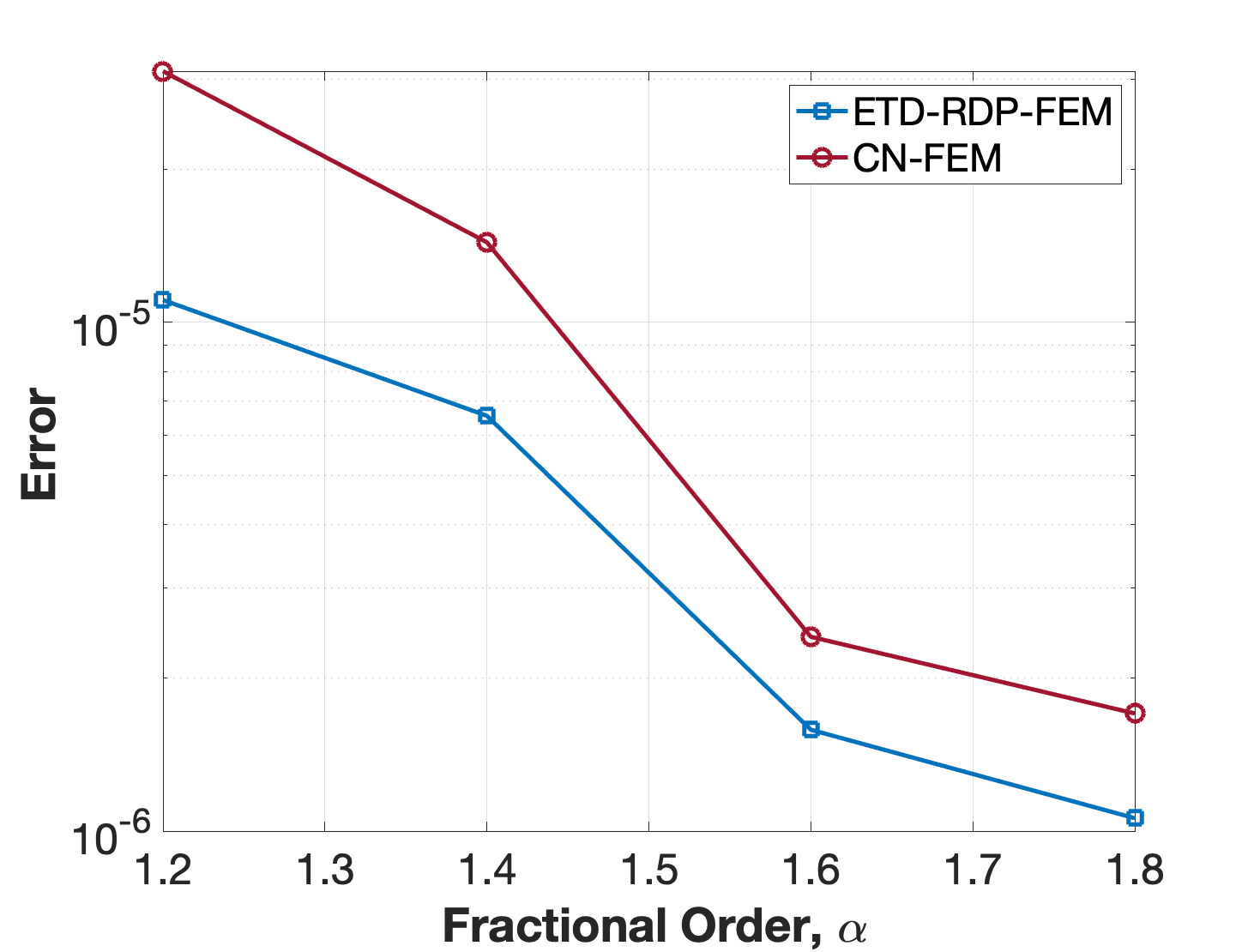}
\caption*{$(a)$ Linear Model}\label{FigureExample1OrderErr1}
\endminipage\hfill
\minipage{0.33\textwidth}
\includegraphics[width=\linewidth]{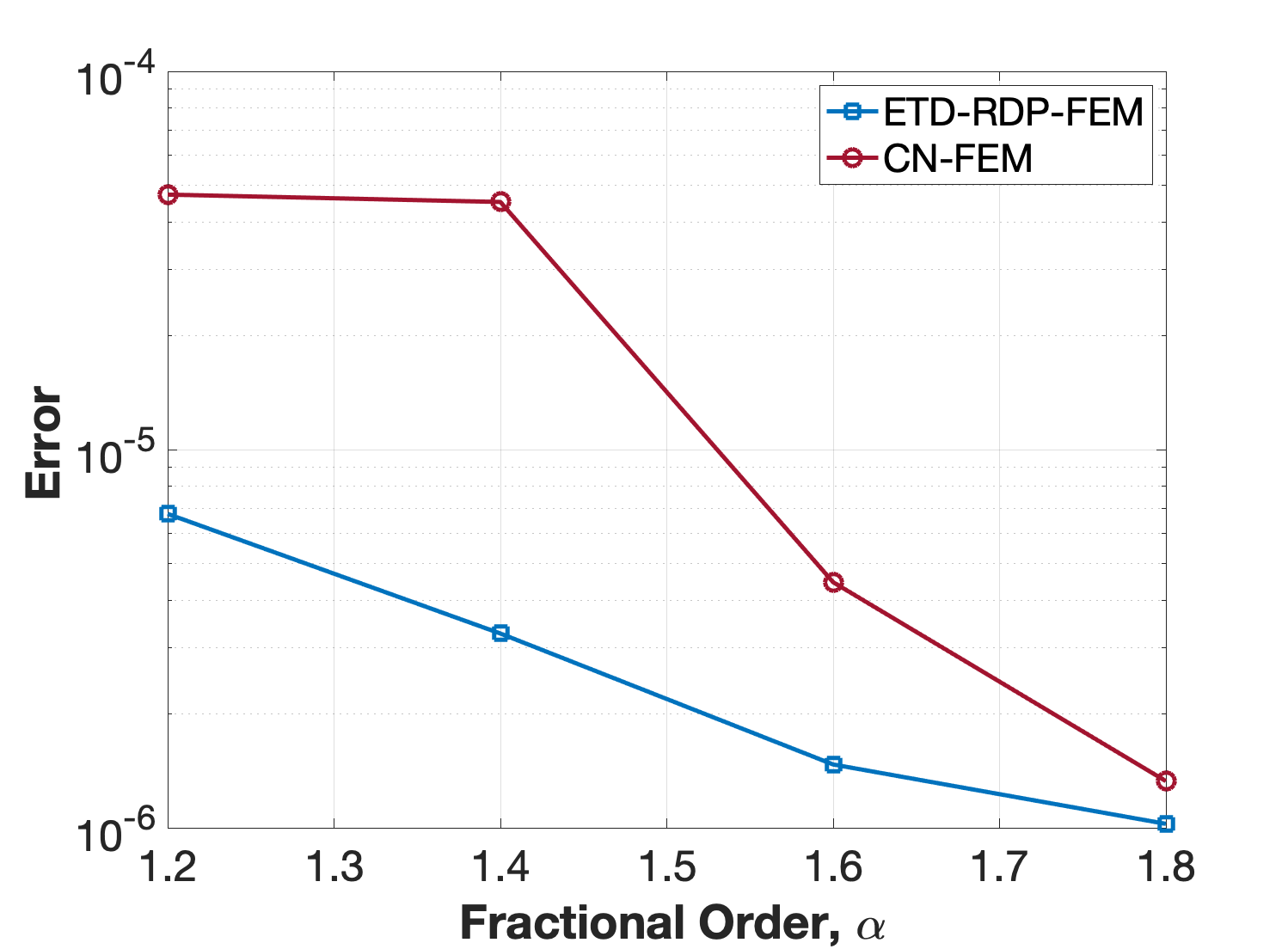}
\caption*{$(b)$ Nonlinear Model}\label{FigureExample2OrderErr1}
\endminipage
\hfill
\minipage{0.33\textwidth}
\includegraphics[width=\linewidth]{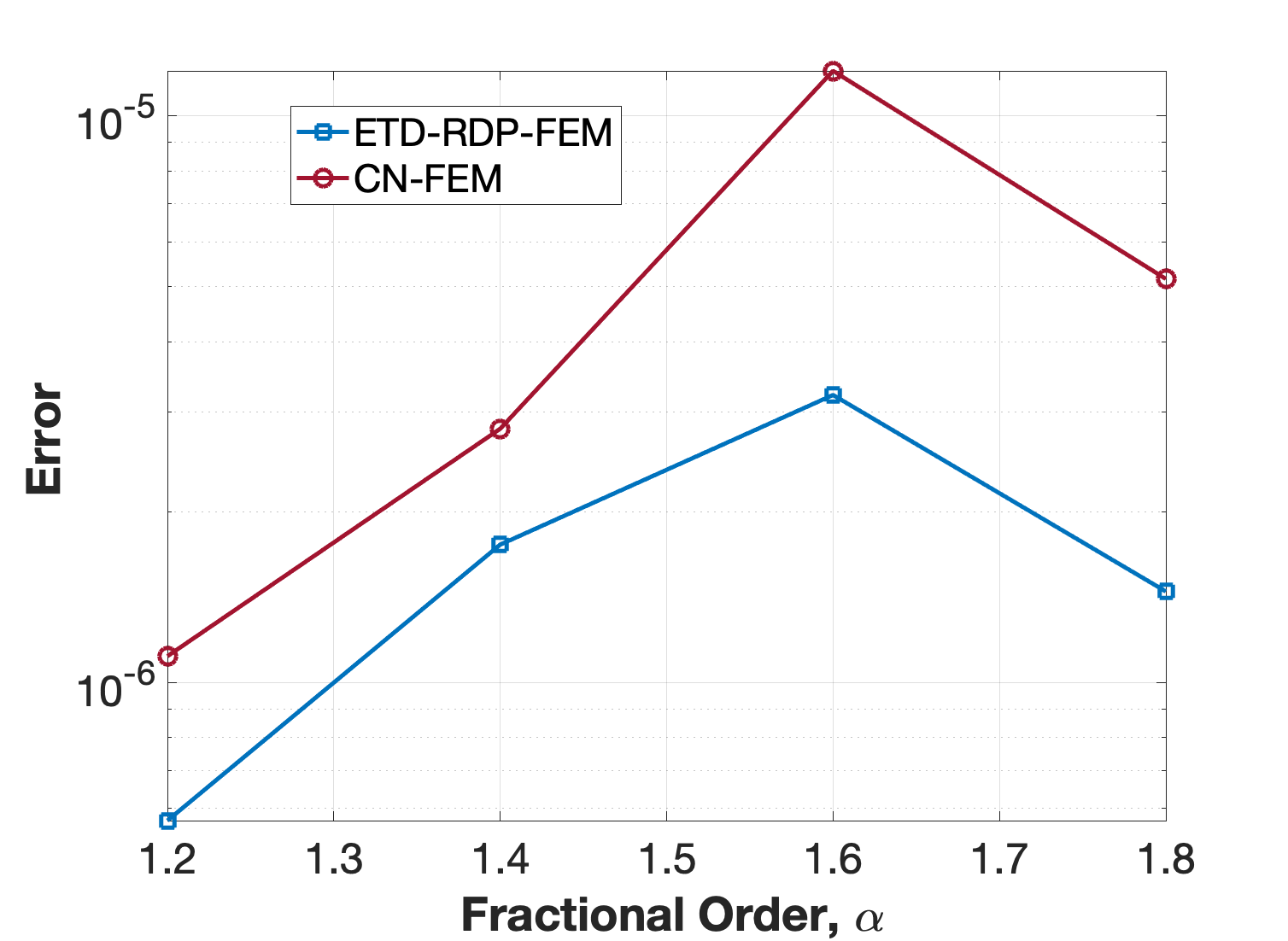}
\caption*{$(c)$ Nonsmooth Model}\label{FigureExample3OrderErr}
\endminipage
\caption{Computational accuracy Plots for Examples \eqref{Example1} - \eqref{Example3}}\label{accuracyplots}
\end{figure}
% \begin{figure}[H]
% \centering
% %\includegraphics[height=8cm,width=12cm]{PercentageImprovement.png}
% \caption{\small Trend of the accuracy of ETD-RDP-FEM and CN-FEM for different values of $\alpha$ in all test problems.}
% \end{figure}
%%%%%%%%%%%%%%%%%%%%%%%%%%%%%%%%%%%%%%%%%%%%%%%%%%%%%%%%%%%%%%%%%%%%%%%%%%%%%%%%%%%%%%%%%%%%%%%%%%%%%%%%%%%%%%%%%%%%%%%%%%%%%%%%%%%%%%%%%%%%%%%%%%%%%%%%%%%%%%%%%%%%%%%%%%%%%%%%%%%%%%%%%%%%%%%%%%%%%%%%%%%%%%%%%%%%%%%%%%%%%%%%%%%%%%%%%%%%%%%%%%%%%%
\section{Conclusion and Recommendation}\label{conclude}
\noindent
In conclusion, the Fractional Diffusion Equation (FDE) serves as a mathematical model that characterizes anomalous transport processes with non-local and long-range dependencies, diverging from traditional diffusion patterns. This makes its numerical solution challenging, as it requires complex integral operators and incurs high computational expenses for precise approximations. In this work, we introduced an Exponential Time Differencing Finite Element Method (ETD-RDP-FEM) to efficiently address both linear and nonlinear problems. Our method, which is grounded on a rational function with distinct real poles to discretize matrix exponentials, leads to an L-stable scheme. The proposed technique exhibits second-order convergence and resiliency when applied to problems with non-smooth initial conditions, underscoring its efficacy and adaptability in handling intricate scenarios. Our findings indicate that the developed approach surpasses the Crank-Nicolson technique in terms of computational efficiency, as demonstrated by reduced CPU time and increased accuracy.\\

In the future, the ETD-FEM scheme may be a promising tool for solving more complex problems, as it offers improved CPU efficiency compared to existing numerical methods. Specifically, future work could focus on extending the ETD-RDP-FEM scheme to solve fractional reaction diffusion equation with an advection term, as well as using it to solve systems of reaction-diffusion-advection equations. Additionally, the scheme's potential applicability to multidimensional problems in irregular domains should also be explored further.
\\
\\
\\

\end{document}